%% file: main.tex
\documentclass[12pt]{amsart}

\numberwithin{equation}{section}
\input{def}
\input{local_commands}

\usepackage{times}

\title[]{Amenable subrelations of treeable equivalence relations and applications to anti-treeability}

\author[]{Anush Tserunyan}
\address[Anush Tserunyan]{Department of Mathematics and Statistics, McGill University, Montreal, QC, Canada}
\email{anush.tserunyan@mcgill.ca}

\author[]{Robin Tucker-Drob}
\address[Robin Tucker-Drob]{Department of Mathematics, University of Florida, Gainesville, FL, USA}
\email{r.tuckerdrob@ufl.edu}



\date{\today}

\begin{document}

\begin{abstract}
We study the structure of amenable subrelations of treeable measure-class-preserving equivalence relations, applying the Radon--Nikodym topography perspective initiated in our companion paper \cite{T-TD:mcp_Adams}.
For a nowhere smooth amenable subrelation, we show that the ends selected by each class in Adams's end selection theorem are precisely those which are nonvanishing along that class.
From this, we derive a number of structural results about nowhere smooth amenable subrelations of treeable equivalence relations.
For example, we show that every nowhere smooth amenable subrelation is contained in an essentially unique maximal amenable subrelation.

This yields an anti-treeability criterion, which we use to show that every nonamenable locally compact second countable group with noncompact amenable radical is not treeable, and that every nonamenable nonunimodular treeable group has a compact open subgroup.
\end{abstract}

\maketitle

\tableofcontents

\subfile{intro}

\subfile{prelims}

\subfile{retraction_to_hyperfinite_along_tree}

\subfile{amenable_subrelations}

\subfile{nontreeability}



\def\MR#1{}
\bibliographystyle{alphaurl-arxiv}
\bibliography{refs}

\end{document}

%% file: def.tex
\usepackage{kpfonts}

\usepackage{amsthm}

\usepackage{amsmath}

\usepackage{amssymb}

\usepackage{dsfont} 

\usepackage{thmtools}

\usepackage{tikz}
\usetikzlibrary{
decorations.markings,
decorations.pathmorphing,
arrows.meta,
positioning,
calc
}

\usepackage{cancel}

\usepackage{verbatim}

\usepackage[shortlabels]{enumitem}
\setlist[enumerate]{leftmargin=*,font=\upshape,align=parleft,label=(\alph*),labelsep=0pt}
\setlist[itemize]{leftmargin=*,labelwidth=*}

\usepackage{graphicx}
\usepackage{float}
\usepackage{subcaption}

\usepackage{color}

\usepackage{xspace}

\usepackage{calrsfs}

\DeclareSymbolFont{defaultmathcal}{OMS}{zplm}{m}{n}
\DeclareSymbolFontAlphabet{\mathcal}{defaultmathcal}

\DeclareSymbolFont{handwritten}{OMS}{rsfs}{m}{n}
\DeclareSymbolFontAlphabet{\handcal}{handwritten}

\definecolor{darkgreen}{RGB}{54,124,50}
\usepackage{hyperref}
\hypersetup{
	pdfstartview={XYZ null null 1.00}, 
	pdfpagemode=UseNone, 
	colorlinks,
	breaklinks,
	linkcolor=blue,
	urlcolor=blue, 
	anchorcolor=blue,
	citecolor=darkgreen,
}

\usepackage[capitalise]{cleveref}

\usepackage{xparse} 

\usepackage{geometry}
\usepackage{microtype}

\input{"commands.tex"}

\input{"environments.tex"}

\usepackage{subfiles}

\usepackage{mdframed}
\newmdenv[
leftmargin = 1cm,
rightmargin = 0pt,
skipabove = 8pt,
skipbelow = 3pt,
innerleftmargin = 8pt,
innertopmargin = 0pt,
innerbottommargin = 0pt,
innerrightmargin = 0pt,
linewidth = 3pt,
topline = false,
rightline = false,
bottomline = false
]{leftbar}

\newcommand{\red}[1]{{\color{red}#1}}

\definecolor{gris}{RGB}{90,90,90}

\definecolor{vert}{RGB}{7,126,26}

\definecolor{purple}{RGB}{116,0,159}

\makeatletter
\def\@settitle{\begin{center}%
		\baselineskip14\p@\relax
		\bfseries
		\uppercasenonmath\@title
		\@title
		\ifx\@subtitle\@empty\else
		\\[1ex]\uppercasenonmath\@subtitle
		\footnotesize\mdseries\@subtitle
		\fi
	\end{center}%
}
\def\subtitle#1{\gdef\@subtitle{#1}}
\def\@subtitle{}
\makeatother

%% file: commands.tex
\newcommand{\N}{\mathbb{N}}
\newcommand{\Z}{\mathbb{Z}}

\newcommand{\R}{\mathbb{R}}

\newcommand{\F}{\mathbb{F}}

\renewcommand{\phi}{\varphi}

\DeclareRobustCommand{\rchi}{{\mathpalette\irchi\relax}}
\newcommand{\irchi}[2]{\raisebox{\depth}{$#1\cchi$}}
\let\cchi\chi
\let\chi\rchi

\newcommand{\EC}{\mathcal{E}}

\newcommand{\dom}{\operatorname{dom}}

\newcommand{\im}{\operatorname{im}}

\newcommand{\Pow}{\handcal{P}}

\newcommand{\del}{\partial}

\makeatletter
\def\moverlay{\mathpalette\mov@rlay}

\def\mov@rlay#1#2{\leavevmode\vtop{%
		\baselineskip\z@skip \lineskiplimit-\maxdimen
		\ialign{\hfil$\m@th#1##$\hfil\cr#2\crcr}}}

\newcommand{\charfusion}[3][\mathord]{
	#1{\ifx#1\mathop\vphantom{#2}\fi
		\mathpalette\mov@rlay{#2\cr#3}
	}
	\ifx#1\mathop\expandafter\displaylimits\fi}
\makeatother

\renewcommand{\leq}{\leqslant}
\renewcommand{\le}{\leqslant}

\renewcommand{\geq}{\geqslant}

\newcommand{\actson}{\curvearrowright}

\newcommand{\onto}{\twoheadrightarrow}

\newcommand{\bij}{\xrightarrow{\,\smash{\raisebox{-0.65ex}{\ensuremath{\scriptstyle\sim}}}\,}}

\newcommand*{\defeq}{\mathrel{\vcenter{\baselineskip0.5ex \lineskiplimit0pt \hbox{\scriptsize.}\hbox{\scriptsize.}}}=}

\DeclareMathSymbol{\lqm}{\mathord}{operators}{``}
\DeclareMathSymbol{\rqm}{\mathord}{operators}{`'}

\newcommand{\set}[1]{\left\{ #1 \right\}}

\newcommand{\gen}[1]{\left\langle #1 \right\rangle}

\newcommand{\clint}[1]{\left[ #1 \right]}

\newcommand{\opint}[1]{\left( #1 \right)}

\newcommand{\rest}[1]{\mathord{|_{#1}}}



%% file: environments.tex
\theoremstyle{plain}

\newtheorem{theorem}[equation]{Theorem}
\newtheorem*{theorem*}{Theorem}

\makeatletter
\def\@empty{}
\def\ifemptycredit#1{%
	\def\tmp{#1}%
	\ifx\tmp\@empty%
	\else%
	{~(#1)}%
	\fi%
}
\makeatother

\newenvironment{namedthm*}[2][]{
	\par\medskip\noindent \textbf{#2}\ifemptycredit{#1}\textbf{.}\itshape\xspace
}{}

\crefname{prop}{Proposition}{Propositions}
\newtheorem{prop}[equation]{Proposition}
\newtheorem*{propo*}{Proposition}

\crefname{property}{Property}{Properties}

\newtheorem*{property*}{Property}

\newtheorem{lemma}[equation]{Lemma}
\newtheorem*{lemma*}{Lemma}

\crefname{claimlemma}{Claim}{Claims}

\crefname{cor}{Corollary}{Corollaries}
\newtheorem{cor}[equation]{Corollary}
\newtheorem*{cor*}{Corollary}

\crefname{obs}{Observation}{Observations}
\newtheorem{obs}[equation]{Observation}
\newtheorem*{obs*}{Observation}

\crefname{obss}{Observations}{Observations}

\newtheorem{obss*}{Observations}

\crefname{fact}{Fact}{Facts}

\newtheorem*{fact*}{Fact}

\theoremstyle{definition}

\crefname{defn}{Definition}{Definitions}
\newtheorem{defn}[equation]{Definition}
\newtheorem*{defn*}{Definition}

\newenvironment{defn**}[1][]{\par\medskip\noindent \textbf{Definition\xspace#1.}\xspace}{}

\crefname{question}{Question}{Questions}
\newtheorem{question}[equation]{Question}
\newtheorem*{question*}{Question}

\crefname{conj}{Conjecture}{Conjectures}

\newtheorem*{conj*}{Conjecture}

\crefname{example}{Example}{Examples}
\newtheorem{example}[equation]{Example}
\newtheorem*{example*}{Example}

\crefname{examples.plain}{Examples}{Examples}
\newtheorem{examples.plain}[equation]{Examples}
\newtheorem*{examples.plain*}{Examples}

\crefname{workinghyp}{Working Hypothesis}{Working Hypotheses}

\theoremstyle{remark}

\crefname{remark}{Remark}{Remarks}
\newtheorem{remark}[equation]{Remark}
\newtheorem*{remark*}{Remark}

\newenvironment{remarklike*}[2][]{\par\medskip\noindent \textit{#2}#1\textbf{.}\rmfamily\xspace}{\smallskip}

\crefname{claim+}{Claim}{Claims}
\newtheorem{claim+}[equation]{Claim}

\crefname{claim}{Claim}{Claims}
\newtheorem{claim}{Claim}
\newtheorem*{claim*}{Claim}

\crefname{subclaim}{Subclaim}{Subclaims}

\newtheorem*{subclaim*}{Subclaim}

\newenvironment{case*}[1]{\smallskip\par\noindent \textit{Case}:~#1.\rmfamily}{}

\crefname{notation}{Notation}{Notations}

\newtheorem*{notation*}{Notation}

\newtheorem*{terminology*}{Terminology}

\crefname{convention}{Convention}{Conventions}
\newtheorem{convention}[equation]{Convention}
\newtheorem*{convention*}{Convention}
\newtheorem*{conventions*}{Conventions}

\crefname{spec}{Speculation}{Speculations}

\newtheorem*{spec*}{Speculation}

\crefname{caution}{Caution}{Cautions}

\newtheorem*{caution*}{Caution}

\crefname{hypothesis}{Hypothesis}{Hypotheses}

\newtheorem*{hypothesis*}{Hypothesis}

\crefname{assumption}{Assumption}{Assumptions}

\newtheorem*{assumption*}{Assumption}

\newcommand{\fntsz}[1][11]{\fontsize{#1}{#1}\selectfont}

\crefname{examples}{Examples}{Examples}

\newenvironment{examples*}[1][\alph*]
{
	\refstepcounter{equation}
	\medskip
	\noindent\textbf{Examples.}
	\medskip
	\begin{enumerate}[\bfseries(\theequation.#1),ref=(\theequation.#1),itemsep=5pt]
}
{
	\end{enumerate}
	\smallskip
}

\theoremstyle{remark}

\declaretheoremstyle[
spaceabove=\topsep, 
spacebelow=6pt,
headfont=\normalfont\itshape,
notefont=\normalfont, notebraces={(}{)},
bodyfont=\normalfont,
postheadspace=4pt,
qed=\mbox{\smaller[4]$\boxtimes$}
]{claimproofstyle}
\declaretheorem[name={Proof of Claim}, style=claimproofstyle, unnumbered]{pf}

\crefname{subsection}{Subsection}{Subsections}

\crefformat{footnote}{#2\footnotemark[#1]#3}

\crefrangeformat{enumi}{#3#1#4--#5#2#6}

\theoremstyle{plain}

%% file: intro.tex
\section{Introduction}

This article is a companion to \cite{T-TD:mcp_Adams}, where we study the \emph{Radon--Nikodym topography} of treed measure-class-preserving (mcp) equivalence relations, i.e.,
the interaction between the componentwise geometry and the system of relative weights supplied by the Radon--Nikodym cocycle.
A \textit{treeing} of a countable Borel equivalence relation $E$ is an acyclic Borel graph whose connected components are precisely the $E$-classes; $E$ is called \emph{treeable} if it admits a treeing, and $E$ is said to be \emph{treed} if it comes equipped with a designated treeing.

The article \cite{T-TD:mcp_Adams} introduces the notion of a \emph{vanishing end}, i.e., an end toward which the Radon--Nikodym cocycle values tend to zero, and uses it to extend the classical Adams dichotomy \cite{Adams90} to the mcp setting:
a treed mcp relation is amenable if and only if almost every component has at most two nonvanishing ends, while the relation is nowhere amenable if and only if in almost every component, the nonvanishing ends form a nonempty closed and perfect set.

This dichotomy characterizes amenability, a global measurable property, in terms of the componentwise Radon--Nikodym topography.
Here we apply this topographical perspective to study the structure of amenable subrelations of treeable mcp equivalence relations.

For a nowhere smooth amenable subrelation $F$ of an mcp relation $E$ treed by $T$, Adams's end selection theorem \cite{Adams:indec_treed_eq_rel_end-selection} provides a canonical maximal measurable selection assigning, to each $F$-class $C$, one or two ends of the $T$-component containing $C$. 
This decomposes $F$ into its \emph{parabolic} and \emph{loxodromic} parts (with respect to $T$), according to where the selection consists of one vs.\ two ends of $T$.
In \cref{parabolic_description,loxodromic_description} we give a structural analysis of the parabolic and loxodromic parts of $F$.
The topographical content of this analysis is distilled in \cref{topography_selects_ends} to the slogan: \emph{the Radon--Nikodym topography selects the ends}. 
That is, the ends of $T$ selected by $C$ are precisely the ends of $T$ which are \emph{nonvanishing along $C$}, i.e., the ends toward which the limsup, along $C$, of the cocycle values is nonzero.
This immediately implies the nonobvious fact, new in the general mcp setting, that a parabolic relation cannot contain a loxodromic subrelation (\cref{parabolic_in_parabolic}).

From this topographical analysis we obtain structural results extending known results from the probability-measure-preserving (pmp) to the mcp setting (as well as some results which are new even in the pmp case, e.g., Corollaries \labelcref{convex_hulls}, \labelcref{relation_transversality}, and \labelcref{Transversality_and_conjugacy}).
For example, the following theorem parallels Bowen's result \cite[Theorem 1.3]{Bowen:eq_rel_act_hyperbolic} for hyperbolic pmp countable Borel equivalence relations.

\begin{theorem}[\cref{unique_max_amenable}]
Every nowhere smooth amenable subrelation $F$ of a treeable mcp equivalence relation $E$ is contained in an essentially unique maximal amenable subrelation $\hat{F}$ of $E$.
\end{theorem}

It follows that in an arbitrary mcp countable Borel equivalence relation $E$, if two amenable subrelations $Q$ and $R$ have nowhere smooth intersection and generate a nowhere amenable subrelation, then $E$ is nowhere treeable (\cref{nowhere_smooth_intersection}).
We call the pair $(Q,R)$ an \textit{anti-treeable configuration} for $E$.

We apply this to provide criteria for nontreeability of locally compact groups in \cref{Hjorth_nontreeable,cor:noncomp_amen_norm_nontreeable}.

\begin{cor}[\cref{cor:noncomp_amen_norm_nontreeable}]
Let $G$ be a locally compact second countable group containing a closed noncompact amenable subgroup $H$ whose normalizer $N_G(H)$ is nonamenable.
Then $G$ is not treeable.
\end{cor}

\cref{Hjorth_nontreeable} is stated in \cite[Corollary 0.6]{Hjorth:non-treeability}; in \cref{rem:Hjorth_correction} we discuss a fixable gap in that published argument and its relation to our proof.

Combining these criteria with the structure theory of locally compact groups, we obtain:

\begin{theorem}[\cref{cor:compact_open}] 
Every nonamenable nonunimodular treeable locally compact second countable group has a compact open subgroup.
\end{theorem}

\subsubsection*{Acknowledgments}

We thank Ruiyuan (Ronnie) Chen, Damien Gaboriau, Antoine Poulin, Grigory (Greg) Terlov, and Konrad Wróbel for many discussions related to this article.
We also thank Antoine Poulin for suggesting a version of \cref{normal_amenable=>anti-treeable}.

ATs was supported by NSERC Discovery Grant RGPIN-2020-07120.
RTD was supported by NSF grant DMS-2246684.

\subsubsection*{AI statement.} The research for this article was completed without any use of AI. 
An initial draft of a much larger manuscript containing most of the results of this article was completed in August 2022; that manuscript was broken into three articles, this being the second (the other two being \cite{T-TD:mcp_Adams,TTD:paddleball}). 
In December 2024 we added the material in \cref{section:non-treeability}.

We used ChatGPT 6.0 for stylistic feedback, and to check for typos and terminological inconsistencies in the complete manuscript.
We used Gemini Pro 3.1 for help with figures.

%% file: prelims.tex
\section{Preliminaries}

We use $\N\defeq\set{0,1,\dots}$ and $\Rpos\defeq\set{r\in\R:r>0}$.
All Borel spaces are standard, all measures are Borel measures, and products carry the product Borel structure.
We refer to \cite{bible} for background on standard Borel spaces and to \cite{JKL} for background on countable Borel equivalence relations.

A \dfn{standard $\sigma$-finite measure space} is a pair $(\Omega,\nu)$, where $\Omega$ is a standard Borel space and $\nu$ is a $\sigma$-finite Borel measure on $\Omega$.
A \dfn{standard probability space} is a pair $(X,\mu)$, where $X$ is a standard Borel space and $\mu$ is a Borel probability measure on $X$.

\begin{convention}\label{convention:standard_probability}
    Unless indicated otherwise, $X$ denotes a standard Borel space, and $(X,\mu)$ denotes a standard probability space.
\end{convention}

If $Y\subseteq X$ is a non-null Borel set, then $\mu_Y$ denotes the normalized restriction of $\mu$ to $Y$.
Given a locally compact second countable group $G$ acting in a Borel way on a standard Borel space $X$, we write $E_G^X$ for the orbit equivalence relation.

\subsection{Measure-class-preserving equivalence relations}\label{subsec:mcp_cbers}

A \dfn{countable Borel equivalence relation} on $X$ is a Borel equivalence relation $E\subseteq X^2$ whose classes are countable.
Unless indicated otherwise, by a \dfn{subrelation} of $E$ we mean an equivalence subrelation of $E$.
We write $[x]_E$ for the $E$-class of $x$.
A set $A\subseteq X$ is an \dfn{$E$-complete section} if it meets every $E$-class, equivalently if $[A]_E\defeq\bigcup_{a\in A}[a]_E=X$.

Let $E$ be a countable Borel equivalence relation on a standard $\sigma$-finite measure space $(\Omega,\nu)$.
Given a property $P$ of Borel equivalence relations, we say that $E$ is \dfn{$\nu$-nowhere $P$} if no positive-measure $E$-invariant Borel set $A\subseteq\Omega$ has $E\rest A$ with property $P$.

The measure $\nu$ is \dfn{$E$-invariant} if the left and right fiber measures on $E$ agree, where
\begin{equation}\label{eq:fiber_measures}
\int_E h\,d\nu_E^\ell \defeq \int_\Omega \sum_{y\in[x]_E}h(x,y)\,d\nu(x),
\qquad
\int_E h\,d\nu_E^r \defeq \int_\Omega \sum_{x\in[y]_E}h(x,y)\,d\nu(y)
\end{equation}
for Borel $h:E\to[0,\infty]$.
If $\nu$ is an invariant probability measure, then $E$ is \dfn{probability-measure-preserving} (\dfn{pmp}) on $(\Omega,\nu)$.

We say that $E$ is \dfn{measure-class-preserving} (\dfn{mcp}) on $(\Omega,\nu)$ if $\nu_E^\ell$ and $\nu_E^r$ are equivalent measures; equivalently, the $E$-saturation of every $\nu$-null set is $\nu$-null \cite[Section 8]{Kechris-Miller}.
In this case the Radon--Nikodym derivative $\rho_\nu \defeq d \nu_E^\ell /d \nu_E^r$, called the \dfn{Radon--Nikodym cocycle of $E$ with respect to $\nu$}, is a Borel cocycle, i.e., it satisfies the cocycle identity $\rho_\nu(z,y)\rho_\nu(y,x)=\rho_\nu(z,x)$ whenever $x,y,z$ are pairwise $E$-related.
Strictly speaking, $d \nu_E^\ell/d \nu_E^r$ only determines an equivalence class of functions, but we can--and shall--always choose a representative $\rho_\nu$ which is indeed a Borel cocycle.

For $x\in \Omega$, $y\in[x]_E$, and $D\subseteq[x]_E$, put
\[
    \rho_\nu^x(y)\defeq \rho_\nu(y,x),
    \qquad
    \rho_\nu^x(D)\defeq \sum_{z\in D}\rho_\nu^x(z).
\]
Thus, $\rho_\nu^x$ is an assignment of relative weights to $[x]_E$ viewed from the basepoint $x$; changing the basepoint rescales all weights by a positive constant.
We write $\rho_\nu^\bullet$ when a statement is not sensitive to changing the basepoint.
For example, a set of pairwise $E$-related points is \dfn{$\rho_\nu$-finite} or \dfn{$\rho_\nu$-infinite} according to whether its $\rho_\nu^\bullet$-weight is finite or infinite.
Finally, $\rho_\nu$ is a \dfn{coboundary} if there is a Borel $m:X\to\Rpos$ such that $\rho_\nu(y,x)=m(y)/m(x)$ for all $(y,x)\in E$.

\begin{convention}\label{convention:rho}
When there is no risk of confusion, we denote by $\rho$ (rather than $\rho_\mu$) the Radon--Nikodym cocycle associated to an mcp countable Borel equivalence relation $E$ on a standard probability space $(X,\mu )$.
\end{convention}

The following smoothness criterion is due to Miller \cite[Proposition 2.1]{Miller:meas_with_cocycle_II}, and independently Kaimanovich \cite[Theorem 23]{Kaimanovich:Hopf_decomp}.

\begin{lemma}\label{smooth=rho-finite}
Let $E$ be an mcp countable Borel equivalence relation on a standard probability space $(X,\mu)$.
Then $E$ is smooth on a conull subset of $X$ if and only if a.e.\ $E$-class is $\rho$-finite.
\end{lemma}

\subsection{Graphs, Treeings, and Ends}

For a graph $G$ on a vertex set $V$, let $E_G$ be the connectedness equivalence relation.
The graph is \dfn{locally countable} if every vertex has countably many neighbors.
For $Y\subseteq V$ we write $G\rest Y$ for the induced subgraph of $G$ on $Y$, and we say that $Y$ is \dfn{$G$-connected} if this induced subgraph is connected.
We say that an equivalence relation on $V$ is \dfn{$G$-connected} every all of its equivalence classes are $G$-connected. 
We write $d_G$ for graph distance and use the standard notions of geodesic paths, rays, and lines.

A graph on a standard Borel space is \dfn{Borel} if it is Borel as a subset of $X^2$.
A locally countable Borel graph $G$ on $(X,\mu)$ is \dfn{mcp} if $E_G$ is mcp, in which case $\rho$ denotes the Radon--Nikodym cocycle of $E_G$.
A Borel graph $G$ is a \dfn{graphing} of $E$ if $E_G=E$.
The equivalence relation $E$ is called \dfn{treeable} if it admits an acyclic graphing $T$, called a \dfn{treeing} of $E$.
In this case we say that $E$ is \dfn{treed} by $T$.

\subsubsection{Lexicographically least paths and nearest points}\label{subsec:lex-least}

Let $G$ be a locally countable Borel graph on $X$.
We fix a Borel proper edge coloring of the edges of $G$ by $\N$, whose existence follows from the Feldman--Moore theorem \cite{Feldman-Moore}, (see also \cite[Theorem 3.1]{Kechris-Miller}).
For $x_0\in X$ and $A\subseteq X$, the \dfn{lexicographically least $G$-path} from $x_0$ to $A$ is the shortest $G$-path from $x_0$ to $A$ whose color sequence is lexicographically least among all shortest such paths; its endpoint is the \dfn{lexicographically nearest point} of $A$ to $x_0$.

\begin{obs}\label{lexproperties}
Suppose $(x_0,x_1,\dots,x_n)$ is the lexicographically least $G$-path from $x_0$ to $A$, and $B\subseteq A$ contains $x_n$.
Then the same path is the lexicographically least $G$-path from $x_0$ to $B$, and for each $0\leq i\leq n$, the tail $(x_i,\dots,x_n)$ is the lexicographically least $G$-path from $x_i$ to $B$.
\end{obs}
While the definitions of lexicographically least paths and nearest points depend on a choice of Borel edge coloring, we will use these constructions without explicit reference to a background choice of coloring. 

\subsubsection{Ends and geodesics in acyclic graphs}\label{subsubsec:geom_acyclic_graphs}

Let $T$ be an acyclic graph on $V$.
Let $\Rays(T)$ denote the set of geodesic rays through $T$.
An \dfn{end} of $T$ is a tail-equivalence class of geodesic rays.
We write $\del_T V$ for the set of ends of $T$ and $\-V^T\defeq V\cup\del_T V$ for the associated end completion.
For $Y\subseteq V$, $\del_TY$ denotes the set of ends represented by rays in $Y$.

Given vertices $x$ and $y$, and ends $\xi$ and $\eta$ belonging to the same $T$-component, we use notation such as $[x,y]_T$, $(x,y]_T$, $[x,\xi)_T$, and $(\xi ,\eta )_T$ for the set of vertices on the geodesic between two points, with the choice of parentheses indicating whether or not a given end-point is included.
For $Y\subseteq\-V^T$, let $\Conv_TY$ be the set of vertices lying on a $T$-geodesic between two points of $Y$.
A subset $Y\subseteq V$ is \dfn{$T$-convex} if $\Conv_TY=Y$.

For a directed edge $e$ of $T$, let $\orig(e)$ and $\term(e)$ be its origin and terminus respectively, and let $e^{-1}$ be the reverse edge.
The \dfn{outgoing edge-boundary} of a subset $Y$ of $V$ is the set of all directed edges $e$ of $T$ with $\orig (e)\in Y$ and $\term (e)\not \in Y$.
A \dfn{half-space} of $T$ is a $T$-connected subset $U$ of $V$ with exactly one directed edge in its outgoing edge-boundary.
For a directed edge $e$ of $T$ we define $\Vo(e)$ to be the unique half-space whose outgoing edge-boundary contains $e$, and $\Vt(e)\defeq\Vo(e^{-1})$.

Let $T$ be a locally countable acyclic Borel graph on $X$.
For $\kappa\in\N\cup\{\aleph_0,2^{\aleph_0}\}$, we say that $T$ is \dfn{$\kappa$-ended} if each component has exactly $\kappa$ ends.
It is \dfn{essentially $\kappa$-ended} if $T\rest{Y}$ is $\kappa$-ended for some Borel $T$-convex $E_T$-complete section $Y\subseteq X$.
We also say that $T$ is \dfn{essentially finitely ended} if there exists a Borel $T$-convex $E_T$-complete section $Y \subseteq X$ such that each $T\rest Y$-component has finitely many ends.

\subsubsection{The vertex-removal topology on the end completion}

For a tree $T$ on the vertex set $V$, we use the \dfn{vertex-removal topology} on the end completion as in \cite[Subsection 2.H]{T-TD:mcp_Adams}, which is generated by all subsets of $V$ together with the sets $U\cup\del_T U$, with $U$ ranging over all half-spaces of $T$.
It follows that this topology is zero-dimensional.
It is also not hard to verify that this topology is generated by the sets $C \cup \del _T C$, where $C$ ranges over the $T$-components of $V$ after the removal of all edges incident to a finite set of vertices (which justifies the name of the topology).

However, when $T$ is not locally finite, the vertex-removal topology is not compact, although it is still Polish if $V$ is countable.
In addition, $\del_T V$ is closed in this topology. 

For a subset $A$ of $\-V^T$ we write $\cl{v}{T}(A)$ for the closure of $A$ in the vertex-removal topology.

\subsubsection{Treeings generated by functions}

Let $f:X\to X$ be a countable-to-one Borel function that is \dfn{acyclic}, i.e., $f^n(x)\neq x$ for all $n\geq 1$ and $x\in X$.
Associated to such an $f$ is the acyclic Borel graph $T_f$, with an undirected edge between $x$ and $f(x)$ for every $x \in X$.
We write $E_f$ for $E_{T_f}$.

In each $T_f$-component, the ray $(f^n(x))_{n\in\N}$ represents the \dfn{forward $f$-end} in that component; all other ends of that component are \dfn{back $f$-ends}.
We say that $f$ is $n$-ended, or essentially $n$-ended, according to whether $T_f$ is.

\subsection{Amenability}\label{subsec:amenability}

\subsubsection{$\mu$-amenability} Let $E$ be an mcp countable Borel equivalence relation on a standard probability space $(X,\mu )$.
Since the measures $\mu_E^{\ell}$ and $\mu _E^r$ are equivalent, $L^{\infty}(E,\mu_E^{\ell})$ and $L^{\infty}(E,\mu_E^r)$ are one and the same algebra, which we simply denote by $L^{\infty}(E)$. 
Similarly, we write $L^{\infty}(X)$ for the algebra $L^{\infty}(X,\mu )$ when the underlying measure class is clear from the context.

The \dfn{(Borel) full semigroup} of $E$, denoted by $[[E]]$, is the set of all Borel partial injections whose graph is contained in $E$.
This forms an inverse semigroup under composition.
We equip $L^{\infty}(E)$ and $L^{\infty}(X)$ with the left actions of $[[E]]$ defined for $\psi\in [[E]]$ by 
\[
(\psi . h)(x,y)\defeq 
\begin{cases}
h(\psi^{-1}(x),y)   &\text{if }x\in\im (\psi ) \\
0                   &\text{otherwise}
\end{cases}
\]
for $h\in L^\infty (E)$ and $(x,y)\in E$, and
\[
(\psi .f)(x)\defeq
\begin{cases}
f(\psi^{-1}(x))   &\text{if }x\in\im (\psi ) \\
0                   &\text{otherwise}
\end{cases}
\]
for $f\in L^\infty (X)$ and $x\in X$.

The equivalence relation $E$ is called \dfn{$\mu$-amenable} if there exists a linear map from $L^{\infty}(E)$ to $L^{\infty}(X)$ that is unital, positive, and $[[E]]$-equivariant.
By \cite{Connes-Feldman-Weiss} (see also \cite[Section 4.8]{Kerr-Li:book}), $E$ being $\mu$-amenable is equivalent to the existence of an $E$-invariant Borel conull set $X_0 \subseteq X$ and a sequence $(m^k)_{k \in \N}$ of Borel assignments $x \mapsto m^k_x$, of a probability vector $m^k_x\in \ell ^1([x]_E)$ to each $x \in X_0$, satisfying 
\[
\lim_{k\to\infty} \| m^k_x - m^k_y\|_1 =0
\]
for all $(x,y)\in E \rest{X_0}$.

\subsubsection{Borel 1- and 2-amenability}
In the purely Borel setting, Jackson, Kechris, and Louveau studied the notion of amenability for countable Borel equivalence relations, stratifying it into a hierarchy of properties called \dfn{Borel $\alpha$-amenability}, where $\alpha$ runs over countable ordinals \cite[Definition 2.12]{JKL}.
(The original terminology used in \cite{JKL} is ``$\alpha$-amenability,'' though ``Borel'' is now commonly added to clarify the context.) 
While it remains an open question whether this hierarchy is strict, we will find it convenient to use its first two levels, which we now define.

Let $E$ be a countable Borel equivalence relation on $X$.
A map $m: x\mapsto m_x$, assigning to each $x\in X$ a probability vector $m_x\in \ell ^1([x]_E)$, is called a \dfn{Borel} assignment if the map $E\to [0,1]$, $(x,y)\mapsto m_x(y)$ is Borel.
The relation $E$ is called 
\dfn{Borel 1-amenable} if there exists a sequence $(m^k)_{k\in\N}$ of Borel assignments $x\mapsto m^k_x$, of a probability vector $m^k_x\in \ell ^1([x]_E)$ to each $x\in X$, satisfying 
\[
\lim_{k\to\infty} \| m^k_x - m^k_y\|_1 =0
\]
for all $(x,y)\in E$.
The relation $E$ is called \dfn{Borel 2-amenable} if there exists a sequence $(m^{(k,j)})_{(k,j)\in\N\times \N}$ of Borel assignments $x\mapsto m^{(k,j)}_x$, of a probability vector $m^{(k,j)}_x\in \ell ^1([x]_E)$ to each $x\in X$, satisfying 
\[
\limsup_{k\to\infty}\limsup_{j\to\infty} \| m^{(k,j)}_x - m^{(k,j)}_y\|_1 =0
\]
for all $(x,y)\in E$.

Fix now a Borel probability measure $\mu$ on $X$ whose measure class is $E$-invariant. 
By \cite{Connes-Feldman-Weiss} (see also \cite[Section 4.8]{Kerr-Li:book}), $E$ being $\mu$-amenable is equivalent to $E$ being Borel 1-amenable on an $E$-invariant $\mu$-conull Borel set.
In fact, it follows from the results of \cite{JKL} that the same holds for Borel $\alpha$-amenability in place of Borel 1-amenability.
However, deducing this from \cite{JKL} is somewhat indirect, and we will only need it in the case $\alpha =2$, so we refer the reader to \cite[Proposition 4.1]{EOSS:hyperfiniteness} for a direct proof of this case.

\begin{prop}[{\cite[Proposition 4.1]{EOSS:hyperfiniteness}}]\label{2-amen_implies_mu-amen}
If an mcp countable Borel equivalence relation $E$ on a standard probability space $(X,\mu)$ is Borel 2-amenable, then $E$ is $\mu$-amenable.
\end{prop}

We will use \cref{2-amen_implies_mu-amen} in combination with the notion of Borel 1-coamenability, which we now define.  

\subsubsection{Borel 1-coamenability} 
Let $F \subseteq E$ be countable Borel equivalence relations on a standard Borel space $X$.
We then let $E / F$ be the standard Borel space obtained as the quotient of $E$ by the (smooth) equivalence relation $\tilde{F}$ that identifies two $E$-related pairs of points when their first coordinates are equal and their second coordinates are $F$-related.
For each $x \in X$ we define $(E/F)_x$ to be the set of all $F$-classes contained in the $E$-class of $x$, and we naturally identify $E/F$ with the set of all pairs $(x,C)$ with $x \in X$ and $C\in (E/F)_x$.
We say that $F$ is \dfn{Borel 1-coamenable} in $E$ if there exists a sequence $(p^k)_{k \in \N}$ of Borel assignments $x \mapsto p^k_x$ (i.e., the map $E/F \to [0,1]$, $(x,C) \mapsto p^k_x(C)$, is Borel), of a probability vector $p^k_x \in \ell ^1((E/F)_x)$ to each $x \in X$, satisfying 
\[
\lim_{k \to \infty} \| p^k_x - p^k_y\|_1 =0
\]
for all $(x,y)\in E$.

\begin{lemma}\label{amenable-coamenable}
    Let $F \subseteq E$ be countable Borel equivalence relations on a standard Borel space $X$.
    If $F$ is Borel 1-amenable and Borel 1-coamenable in $E$ then $E$ is Borel 2-amenable.
\end{lemma}
\begin{proof}
    Let $(m^j)_{j\in\N}$ and $(p^k)_{k\in\N}$ be witnesses to Borel 1-amenability and Borel 1-coamenability respectively.
    By Luzin--Novikov uniformization, there exists a Borel function $s : E/F \to X$ so that the map $(x,y) \mapsto (x,s(x,[y]_F))$ is a selector for $\tilde{F}$.
    For $k,j\in \N$ and $x\in X$ define $q^{(k,j)}_x\in \ell ^1([x]_E)$ by 
    \[
    q^{(k,j)}_x \defeq \sum_{C \in (E/F)_x} p^k_x(C) m^j_{s(x,C)}.
    \]
    The triangle inequality gives the required convergence for the sequence $(q^{(k,j)})_{(k,j)\in\N\times\N}$.
\end{proof}

\begin{convention}
Except when Borel 1-amenability, Borel 1-coamenability, or Borel 2-amenability are explicitly mentioned, the relevant notion of amenability in this article is $\mu$-amenability.
Therefore, when the measure $\mu$ is clear from the context, and no confusion is likely, we will use ``amenable'' to mean ``$\mu$-amenable.'' 
\end{convention}

\subsection{End selection and parabolic/loxodromic subrelations}

Let $T$ be a treeing of a countable Borel equivalence relation $E$ on $X$.
Given a Borel subset $X_0$ of $X$, a \dfn{$T$-end selection} on $X_0$ is a map $x\mapsto\EC(x)$ assigning to each $x\in X_0$ a finite nonempty set of ends of $T\rest{[x]_E}$.
The end selection is said to be \dfn{Borel} if it admits a Borel lift to the standard Borel space of finite subsets of $\Rays(T)$.
For a Borel subequivalence relation $F \subseteq E$, the selection is \dfn{$F$-invariant} if $\EC(x)=\EC(y)$ for all $(x,y)\in F\rest{X_0}$.

For an mcp countable Borel equivalence relation $F$ on $(X,\mu)$, set
\[
    \Xns_F\defeq\{x\in X:\rho^x([x]_F)=\infty\}.
\]
By \cref{smooth=rho-finite}, this is the essentially unique largest $F$-invariant Borel set on which $F$ is $\mu$-nowhere smooth.

\begin{theorem}[End selection {\cite[Theorem 3.1]{Adams:indec_treed_eq_rel_end-selection}}]\label{end_selection}
Let $T$ be a treeing of an mcp countable Borel equivalence relation $E$ on $(X,\mu )$, and assume that $E$ is $\mu$-nowhere smooth. 
A Borel subequivalence relation $F$ of $E$ is amenable if and only if, after discarding a null set, there is an $F$-invariant Borel $T$-end selection, $X\rightarrow \Pow _{\mathrm{fin}}(\del_T X)$, $x \mapsto \EC (x)$. 

Moreover, if $F$ is an amenable Borel subequivalence relation of $E$ then
\begin{enumerate}[(a)]
\item every $F$-invariant Borel $T$-end selection on $\Xns_F$ satisfies $1\leq |\EC (x)|\leq 2$ for a.e.\ $x\in \Xns_F$;
\item after discarding a null set, there exists an essentially unique maximum $F$-invariant Borel $T$-end selection $x\mapsto \EC _F (x)$ on $\Xns_F$. 
That is, if $x\mapsto \EC (x)$ is any other $F$-invariant Borel $T$-end selection on a relatively conull subset of $\Xns_F$, then $\EC (x)\subseteq \EC _F(x)$ for a.e.\ $x\in \Xns_F$.
\end{enumerate} 
\end{theorem}

Let $E$ be an mcp countable Borel equivalence relation on $(X,\mu)$ treed by $T$, and assume that $E$ is $\mu$-nowhere smooth.
For an amenable subequivalence relation $F\subseteq E$, let $\EC_F:\Xns_F\to\Pow_{\mathrm{fin}}(\del_TX)$ be the essentially unique maximum $F$-invariant Borel $T$-end selection from \cref{end_selection}.
Set
\[
    \Xpar_F\defeq\{x\in\Xns_F:|\EC_F(x)|=1\},
    \qquad
    \Xlox_F\defeq\{x\in\Xns_F:|\EC_F(x)|=2\}.
\]
These are the \dfn{parabolic} and \dfn{loxodromic} sets of $F$ in $T$.
The relation $F$ is \dfn{parabolic} in $T$ if $\Xpar_F$ is conull, and \dfn{loxodromic} in $T$ if $\Xlox_F$ is conull.
For $x\in\Xpar_F$, write $\xi_F(x)$ for the unique element of $\EC_F(x)$.
If $F$ is amenable and ergodic with $\Xns_F$ conull, then $F$ is either parabolic or loxodromic in $T$.
When $T=T_f$, we use the same adjectives for $f$ according to whether $E_f$ is parabolic or loxodromic in $T_f$.

\subsection{Nonvanishing ends and the generalized Adams dichotomy}\label{subsec:nonvanishing_along}

Let $T$ be an acyclic locally countable mcp Borel graph on $(X,\mu)$, with Radon--Nikodym cocycle $\rho$.
An end $\eta\in\del_T[x]_{E_T}$ is \dfn{$\rho$-vanishing} if $\limsup_{y\to\eta}\rho^x(y)=0$, and \dfn{$\rho$-nonvanishing} otherwise.
This does not depend on the chosen basepoint $x$ in the component. 
We suppress the cocycle $\rho$ in the terminology when it is clear from the context.

For $D\subseteq[x]_{E_T}$, the end $\eta$ is \dfn{vanishing along $D$} if
\[
    \limsup_{\substack{y\to\eta\\ y\in D}}\rho^x(y)=0,
\]
with the convention that the limsup is $0$ when $\eta$ is not an accumulation point of $D$ in the end completion.
Otherwise, $\eta$ is said to be \dfn{nonvanishing along $D$}.
Usually $D=[x]_F$ for a subrelation $F\subseteq E_T$, in which case we speak of vanishing or nonvanishing ``along the $F$-class of $x$.''

We close the preliminaries by recalling the trichotomy form of the generalized Adams dichotomy from \cite{T-TD:mcp_Adams}.

\begin{theorem}[Trichotomy for acyclic mcp graphs]\label{topography:trichotomy}
Let $T$ be a locally countable acyclic mcp Borel graph on $(X,\mu)$.
\begin{enumerate}[(a)]
\item\label{topography:trichotomy:smooth} $E_T$ is smooth on a conull set if and only if all ends of almost every $T$-component are vanishing.
\item\label{topography:trichotomy:amenable} $E_T$ is amenable and $\mu$-nowhere smooth if and only if almost every $T$-component has exactly one or two nonvanishing ends.
\item\label{topography:trichotomy:nowhere_amenable} $E_T$ is $\mu$-nowhere amenable if and only if, in almost every $T$-component, the set of nonvanishing ends is nonempty, closed, and has no isolated points.
\end{enumerate}
\end{theorem}

%% file: retraction_to_hyperfinite_along_tree.tex
\section{Finite-to-one retractions along a treeing}\label{sec:retractions}

The goal of this section is, for an acyclic mcp graph $T$, to give a measurable construction of \emph{finite-to-one} retractions to classes of an amenable subequivalence relation of $E_T$.

The theorem is stated relative to a non-null Borel subset $X_0\subseteq X$, rather than only for subrelations defined on all of $X$.
This extra flexibility is needed in applications. 
A given treeing $T$ of $E$ usually does not restrict to a treeing of $E|_{X_0}$, so one cannot in general reduce to the case $X_0=X$ by simply restricting.

\begin{theorem}\label{end-preserving_slide}
Let $T$ be an acyclic locally countable mcp Borel graph on a standard probability space $(X,\mu)$, let $X_0$ be non-null Borel subset of $X$, and let $F$ be an amenable $\mu\rest{X_0}$-nowhere smooth Borel subrelation of $E_T \rest{X_0}$.

Then, after discarding an $E_T$-invariant Borel null set, $F$ admits a Borel treeing $T_F$ such that for each $F$-class $C \subseteq X_0$, the identity map $C \to C$ extends uniquely to a homeomorphism $\cl{v}{T_F}(C) \to \cl{v}{T}(C)$, where both spaces are equipped with the vertex-removal topology.

In particular, if $\rho : E_T \to \Rpos$ is a Borel cocycle, then for each $x \in X_0$, the homeomorphism above restricts to a bijection between the set of $\rho$-nonvanishing ends of $T_F \rest{[x]_F}$ and the set of ends of $T \rest{[x]_{E_T}}$ which are $\rho$-nonvanishing along $[x]_F$.
\end{theorem}

The proof of \cref{end-preserving_slide} is given at the end of this section after some discussion of the strategy and development of necessary ingredients.

A simple observation needed for our discussion is that for a subset $U$ of the set of vertices $X$ of a tree $T$, the closure of $U$ in the vertex-removal topology contains the same ends as the closure of the convex hull $\Conv_T(U)$.
The standard method of retracting $T$ to $U$ is to have each vertex $x \in \Conv_T(U)$ choose the lexicographically nearest point to $x$ in $U$.
This partitions $\Conv_T(U)$ into Voronoi regions centered at points in $U$ which, when collapsed to their centers, give the tree structure on $U$, see \cref{fig:lex-least}.
This construction can also be performed with $U$ varying over the classes of an equivalence relation $F$ on $X$ to obtain a forest whose components are exactly the $F$-classes.

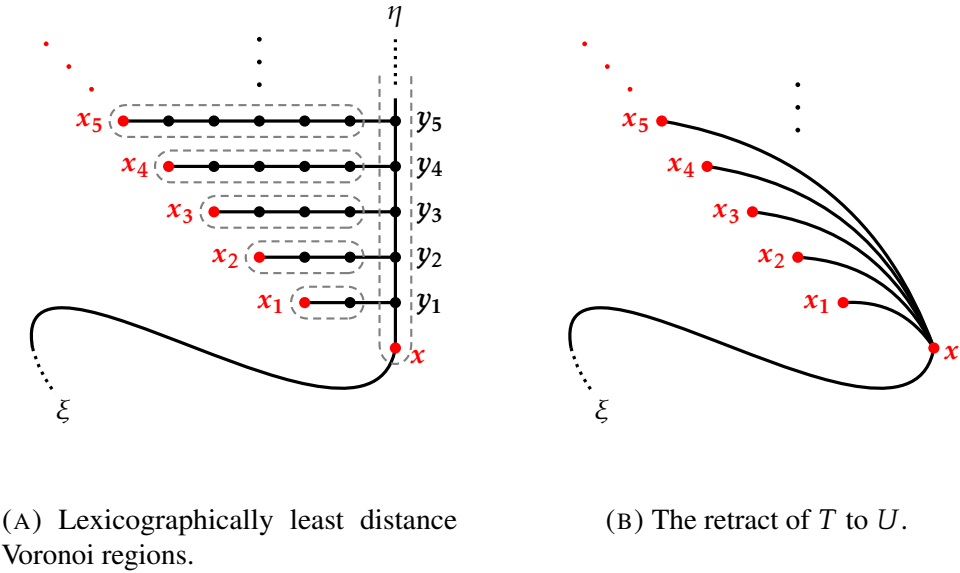
\begin{figure}[htbp]
\subfloat[Lexicographically least distance Voronoi regions.]{%
\input{fig_lex-least_Voronoi}%
}
\hspace{1cm}
\subfloat[The retract of $T$ to $U$.]{%
\input{fig_lex-least_retract}%
}
\caption{
A lexicographically least retraction of a tree $T$ to the set $U$ of red vertices $x, x_1, x_2, x_3, \dots$ resulting in the loss of the end $\eta$ of $T \rest{\Conv_T(U)}$.
}
\label{fig:lex-least}
\end{figure}

However, this may result in a retraction with infinite, or even ill-founded, fibers, which may cause a ``loss'' of an end of $T \rest{\Conv_T(U)}$, as shown in \cref{fig:lex-least}.

This retraction construction is formalized in the Borel setting for treed equivalence relations in \cite[Proposition 3.3(i)]{JKL}.
We use the equivalent language of retraction bundles, which we present in the first part of this section. 
The second part of the section is devoted to a modification of the standard retraction method which uses end selection to produce finite-to-one retractions, thereby preserving the end structure (in the sense of \cref{end-preserving_slide}); see \cref{fig:contrast_with_lex-least}.

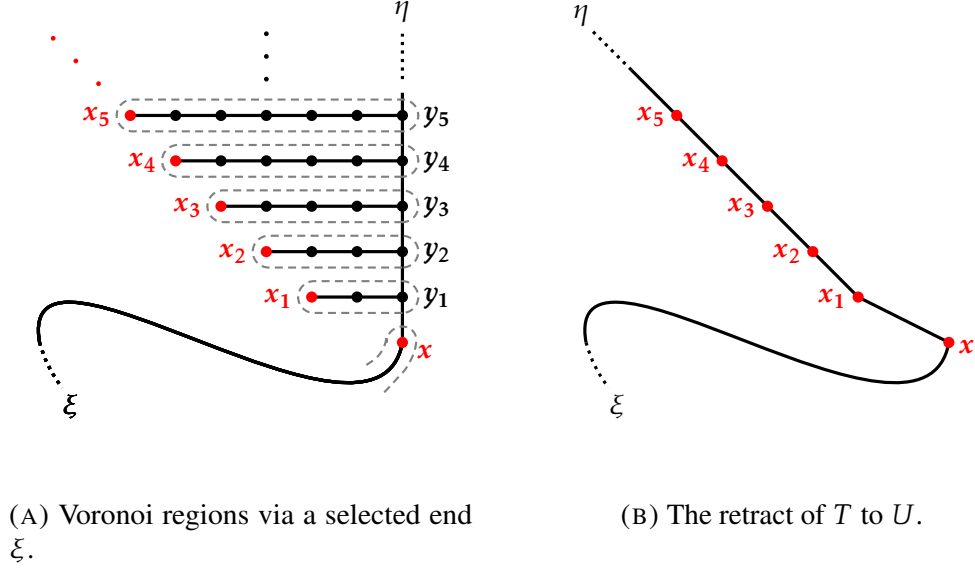
\begin{figure}[htbp]
\subfloat[Voronoi regions via a selected end $\xi$.]{%
\input{fig_contrast_with_lex_Voronoi}
}
\hspace{1cm}
\subfloat[The retract of $T$ to $U$.]{%
\input{fig_contrast_with_lex_retract}%
}
\caption{Modified retraction via a selected end $\xi$}
\label{fig:contrast_with_lex-least}
\end{figure}

It is worth noting, however, that in the pmp setting the issue of having ill-founded fibers does not arise when retracting to a subequivalence relation of finite cost, as the following remark explains.

\begin{remark}[Well-founded fibers in the pmp vs.\ mcp settings]
For an acyclic pmp graph $T$ on $(X,\mu)$ and a subequivalence relation $F \subseteq E_T$ of finite cost, after discarding an $E_T$-invariant null set, for each $F$-class $C$ the standard lexicographically least retraction $r_C:\Conv_T(C)\to C$ has well-founded fibers.
Indeed, since $F$ has finite cost, after discarding an $E_T$-invariant null set, the treeing $T_F$ of $F$ obtained from these retractions (along with $T$) must be locally finite a.e.
It therefore suffices to show that if 
$x\in C$ and $r_C^{-1}(x)$ is ill-founded, then $x$ has infinite degree in $T_F$.

Fix an end $\xi$ of $T$ such that the geodesic ray $[x,\xi )_T$ is contained in $r_C^{-1}(x)$. 
Since $r_C^{-1}(x)\subseteq \Conv_T(C)$, we may recursively find a sequence $(y_n)_{n\in\N}$ in $C$ such that the directed path from $y_n$ to $x$ meets $[x,\xi )_T$ in an initial segment which gets longer with $n$.
Let $z_n$ be the last point on this path that is not in $r_C^{-1}(x)$.
Then the points $r_C(z_n)$, $n\in \N$, are distinct, and they are all adjacent to $x$ in $T_F$.

However, in the mcp setting, the lexicographically least retraction, even to an amenable subrelation, can have ill-founded fibers. For example, we can modify the graph in \cref{example:full_convex_hull} by subdividing, for each point $(x, \gamma)$ with $\gamma\neq e_{\F_e}$, the edge from $(x, \gamma)$ to $(\gamma^{-1}.x, e_{\F_2})$ into a path of length $|\gamma|_{\F_2} + 1$, declaring the $F$-class of each newly added (degree $2$) vertex to be a singleton.
This is why we need to modify the standard lexicographically least construction to obtain finite-to-one retractions, which avoid the end-loss concern as in \cref{fig:lex-least}.
\end{remark}

\subsection{Retraction bundles}

For a graph $G$ on a set $X$ and an equivalence relation $F$ on $X$, the \emph{quotient graph} $G / F$ is the graph on $X/F$ defined by declaring two $F$-classes $C_0$ and $C_1$ to be $G/F$-adjacent if and only if $C_0\neq C_1$ and there exist $x_0\in  C_0$ and $x_1\in C_1$ which are $G$-adjacent.

\begin{lemma}\label{finite_quotient_graph}
Let $T$ be a tree on a set $V$ and let $F$ be a $T$-connected finite equivalence relation on $V$. 
Then the quotient graph $T / F$ is still a tree. 
Furthermore, the quotient map $\pi : V \onto V / F$ extends uniquely to a continuous open map $\-\pi : \-{V}^T \to \-{V/F}^{T/F}$ whose restriction to $\del_T V$ is a homeomorphism $\del_T V \bij \del_{T/F} (V/F)$, where all spaces are equipped with the vertex-removal topology.
\end{lemma}
\begin{proof}
That $T / F$ is a tree follows from the $T$-connectedness of $F$.
It is also clear that $\pi$ takes geodesic rays in $T$ to those in $T / F$ (after removing consecutive repeated vertices), and every geodesic ray in $T / F$ arises in this fashion. 
Furthermore, two geodesic rays in $T$ are tail equivalent if and only if their $\pi$-images in $T / F$ are tail equivalent, so $\pi$ indeed extends to a well-defined surjection $\-\pi : \-V^T \to \-{V / F}^{T / F}$, which restricts to a bijection $\del_T V \to \del_{T / F} (V / F)$.

For the continuity and openness of $\-\pi$ observe that for each end $\xi$ of $T$, the collection of all neighborhoods of $\xi$ whose intersection with $V$ is $F$-invariant forms a neighborhood basis for $\xi$, which descends to a neighborhood basis for $\-\pi(\xi)$.
\end{proof}

\begin{defn}
Let $T$ be an acyclic graph on a set $X$ and let $C$ be a subset of a $T$-component.
A \dfn{retraction} along $T$ to $C$ is a map $g : \Conv_T C \to C$ that restricts to the identity map on $C$, and has the property that $g^{-1}(x)$ is $T$-connected for every $x \in C$. 

Given such a retraction $g$, let $=_g$ be the equivalence relation of having the same image under $g$.
We let $gT$ denote the image of the quotient graph $(T \rest{\Conv_T C}) / =_g$ under the natural identification of $(\Conv_T C) / =_g$ with $C$, so that $gT$ is a tree on $C$, called the \dfn{retract} of $T$ to $C$ induced by $g$.
\end{defn}

\begin{prop}\label{finite-to-one_retract}
Let $T$ be a tree on a set $V$ and let $C$ be a subset of $V$.
Then each finite-to-one retraction $g : \Conv_T(C) \to C$ along $T$ to $C$ extends uniquely to a continuous open map $\-g : \cl{v}{T}(\Conv_T C) \to \cl{v}{gT}(C)$ whose restriction to $\del_T (\Conv_T C)$ is a homeomorphism $\del_T (\Conv_T C) \bij \del_{gT} C$.
\end{prop}
\begin{proof}
Apply \cref{finite_quotient_graph} to $T \rest{\Conv_T C}$ and the equivalence relation $=_g$.
\end{proof}

\begin{defn}
Let $T$ be an acyclic Borel graph on a standard Borel space $X$, let $X_0$ be a Borel subset of $X$, and let $F$ be a Borel subequivalence relation of $E_T\rest{X_0}$.

An \dfn{$F$-invariant retraction bundle along $T$ to $F$}, denoted $(g_x)_{x\in X_0}$, is an $F$-invariant assignment $x \mapsto g_x$ to each $x \in X_0$ of a retraction $g_x : \Conv_T [x]_F \to [x]_F$ along the tree $T \rest{[x]_{E_T}}$ to $[x]_F$. 
We say that this retraction bundle is \dfn{Borel} if the map $(x,y) \mapsto g_x(y)$ is Borel as a map from $\set{(x,y) : x\in X_0 \text{ and }y \in \Conv_T [x]_F}$ to $X_0$.

Let $(g_x)_{x\in X_0}$ be an $F$-invariant retraction bundle along $T$ to $F$. 
For each $F$-class $C$ and each point $x\in C$, the retract $g_x T$, of $T$ to $C$ induced by $g_x$, is then a tree on $C$ that does not depend on the choice of $x \in C$. 
We thus obtain a treeing $g_* T$ of $F$, whose restriction to each $F$-class $[x]_F$ is precisely $g_x T$. 
We call $g_* T$ the \dfn{retract} of $T$ to $F$ induced by the bundle $(g_x)_{x\in X_0}$.
We say that an $F$-invariant retraction bundle $(g_x)_{x\in X_0}$ is \dfn{finite-to-one} if each $g_x$ is a finite-to-one map.
\end{defn}


\begin{cor}\label{retraction=>end-embedding}
Let $T$ be an acyclic locally countable Borel graph on a standard Borel space $X$, let $X_0$ be a Borel subset of $X$, and let $F$ be a Borel subrelation of $E_T \rest{X_0}$. 
If there exists a finite-to-one $F$-invariant Borel retraction bundle along $T$ to $F$, then there exists a Borel treeing $T_F$ of $F$ such that for each $F$-class $C$, the identity map $C \to C$ extends uniquely to a homeomorphism $\cl{v}{T}(C) \bij \cl{v}{T_F}(C)$, where both spaces are equipped with the vertex-removal topology.
\end{cor}
\begin{proof}
Given an $F$-invariant Borel retraction bundle $(g_x)_{x\in X_0}$ along $T$ to $F$, let $T_F \defeq g_* T$ be the induced retract and apply \cref{finite-to-one_retract}.
\end{proof}

\subsection{Constructing finite-to-one retractions using end-selection}

In light of \cref{retraction=>end-embedding}, to prove \cref{end-preserving_slide}, it is enough to obtain a finite-to-one $F$-invariant Borel retraction bundle to $F$ along $T$, and we do so below using a Borel $T$-end selection, along with a Borel lexicographic ordering of paths.

\begin{lemma}[Accumulation on selected ends]\label{chosen_ends_are_spanned}
Let $T$ be an acyclic locally countable Borel graph on a standard Borel space $X$, let $X_0$ be a Borel subset of $X$, and let $F$ be a Borel subequivalence relation of $E_T\rest{X_0}$. 
Suppose that there exists an $F$-invariant Borel $T$-end selection $x \mapsto \EC (x)$ on $X_0$ with $1 \le |\EC(x)| \le 2$.
Then after discarding an $F$-invariant Borel subset of $X_0$ on which $F$ is smooth, every $x \in X_0$ satisfies $\EC (x) \subseteq \del_T (\Conv_T [x]_F) \subseteq \cl{v}{T}([x]_F)$.
\end{lemma}
\begin{proof}
Let $Y$ be the set of all $x \in X_0$ for which $\EC (x)$ is not contained in $\del_T (\Conv_T [x]_F)$. For each $x \in Y$ and $\xi \in \EC (x)$ not belonging to $\del_T (\Conv_T [x]_F)$, there is a unique point $v_\xi\in\Conv_T [x]_F$ closest to $\xi$. Let $V_x = \{ v_{\xi} : \xi \in \EC (x)\setminus \del_T (\Conv_T [x]_F)\}$. Each $V_x$ is nonempty and finite (has at most two elements), and the map $x \mapsto V_x$ is $F$-invariant, so it witnesses the smoothness of $F \rest{Y}$.
\end{proof}

\begin{lemma}[Retraction via one end]\label{retraction_via_1-end_selection}
Let $T$ be an acyclic locally countable Borel graph on a standard Borel space $X$, 
let $X_1$ be a Borel subset of $X$, and let $F$ be a Borel subequivalence relation of $E_T\rest{X_1}$. 
Suppose that there exists an $F$-invariant Borel $T$-end selection $x \mapsto \xi_x : X_1 \to \del_T X$ on $X_1$ in which each $F$-class selects exactly one $T$-end. 

Then there is an $F$-invariant Borel subset $Y$ of $X_1$ on which $F$ is smooth, and a finite-to-one $F$-invariant Borel retraction bundle $x \mapsto g_x$ to $F \rest{X_1 \setminus Y}$ along $T$ such that for each $z\in [x]_F$, the vertices in $g_x^{-1}(z)$ form a finite initial segment of the geodesic ray $[z, \xi_z)_T$ (\cref{fig:one_end_retraction}).
\end{lemma}

\begin{figure}[htbp]
\subfloat[Voronoi tesselation via one selected end $\xi_x$]{%
\input{fig_parabolic_Voronoi}%
}
\hspace{0.5cm}
\subfloat[Retract via one selected end]{%
\input{fig_parabolic_retract}%
}
\caption{
Retraction of $T$ to $[x]_F$ (depicted in red) via one selected end $\xi_x$.
}
\label{fig:one_end_retraction}
\end{figure}
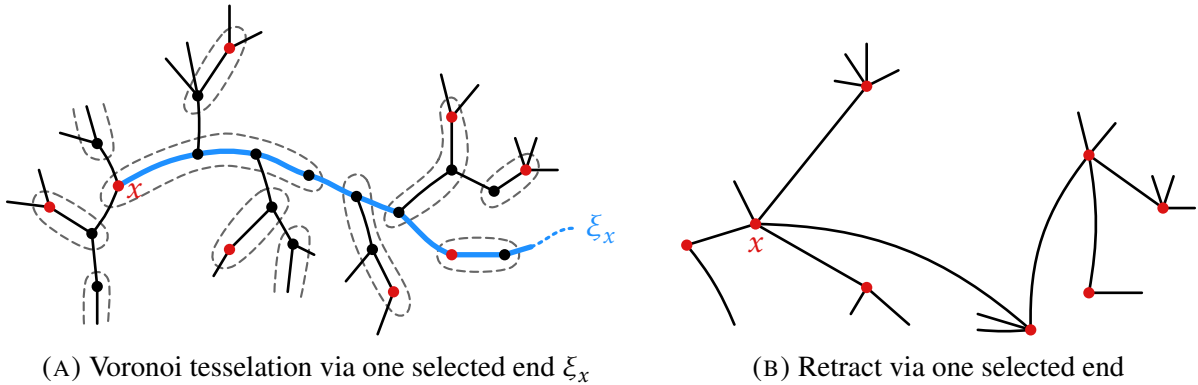

\begin{proof}
By \cref{chosen_ends_are_spanned} we may assume that $\xi_x \in \del_T (\Conv_T [x]_F)$ for all $x \in X_1$. 
For $x\in X_1$ and $y\in\Conv_T [x]_F$ define $g_x (y)$ to be the lexicographically nearest point to $y$ through $T$ among points in $[x]_F$ whose geodesic to $\xi_x$ passes through $y$, see \cref{fig:one_end_retraction}. 
Then $x\mapsto g_x$ is $F$-invariant, and $g_x$ restricts to the identity map on $[x]_F$ for each $x\in X_1$. 
For each $z\in [x]_F$, the preimage $g_x^{-1}(z)$ is contained in $[z,\xi_x )_T$, and it is $T$-convex by \cref{lexproperties}. 
If such a preimage $g_x^{-1}(z)$ is infinite then it contains all of $[z, \xi _x )_T$, which can happen for at most one $z\in [x]_F$. Thus $F$ is smooth on the set $Y$, of all $x\in X_1$ for which there is some $z\in [x]_F$ with $g_x^{-1}(z)$ infinite. 
\end{proof}

\begin{lemma}[Retraction via two ends]\label{retraction_via_2-end_selection}
Let $T$ be an acyclic locally countable Borel graph on a standard Borel space $X$, let $X_2$ be a Borel subset of $X$, and let $F$ be a Borel subequivalence relation of $E_T\rest{X_2}$. 
Suppose that there exists an $F$-invariant Borel $T$-end selection $x \mapsto \EC (x)$ on $X_2$ with $|\EC (x)| = 2$ for each $x \in X_2$.

Then there is an $F$-invariant Borel subset $Y$ of $X_2$ on which $F$ is smooth, and a finite-to-one $F$-invariant Borel retraction bundle $(g_x)_{x\in X_2\setminus Y}$ to $F\rest{X_2\setminus Y}$ along $T$ such that for each $y \in [x]_F$, the restriction of $T$ to $g_x^{-1}(y)$ is a finite tree with at most three leaves, i.e., a possibly degenerate tripod (\cref{fig:two-end_retraction}). 
\end{lemma}

\begin{figure}[htbp]
\subfloat[Voronoi tesselation via the selected line $\ell_x$]{%
\input{fig_loxodromic_Voronoi}%
}
\hspace{0.5cm}
\subfloat[Retract via two ends]{%
\input{fig_loxodromic_retract}%
}
\caption{
Retraction of $T$ to $[x]_F$ via a selected line $\ell_x$ (two selected ends).
}
\label{fig:two-end_retraction}
\end{figure}
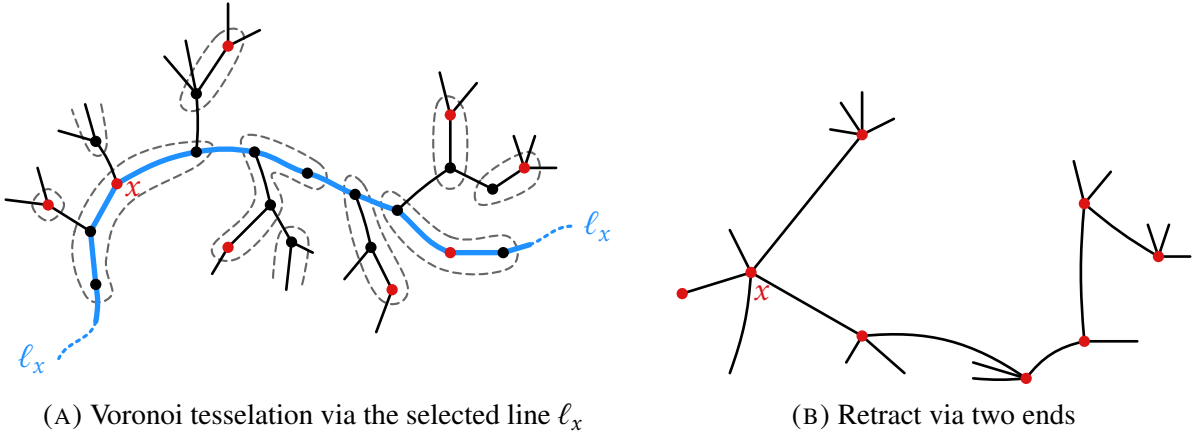

\begin{proof}
For each $x \in X_2$, let $\ell_x$ denote the geodesic line in $T$ determined by the ends in $\EC (x)$. By \cref{chosen_ends_are_spanned} we may assume that $\ell_x \subseteq \Conv_T [x]_F$ for all $x \in X_2$. 
For each $y \in \Conv_T[x]_F$, if $y \in \ell_x$ then let $g_x(y)$ be the lexicographically nearest point to $y$ in $[x]_F$ through $T$; otherwise, let $g_x(y)$ be the lexicographically nearest point to $y$ through $T$ among points in $[x]_F$ whose geodesic to $\ell_x$ passes through $y$, see \cref{fig:two-end_retraction}. 
Then $x\mapsto g_x$ is $F$-invariant, and for each $z\in [x]_F$ the restriction of $T$ to $g_x^{-1}(z)$ is a tree with at most three leaves. 
The preimage $g_x^{-1}(z)$ can be infinite for at most two $z \in [x]_F$, so $F$ is smooth on the set $Y$, of all $x\in X_2$ for which there is some $z\in [x]_F$ with $g_x^{-1}(z)$ infinite. 
\end{proof}

\begin{proof}[Proof of \cref{end-preserving_slide}] 
By \cref{end_selection}, after discarding an $E_T$-invariant Borel null set there exists an $F$-invariant Borel $T$-end selection $x \mapsto \EC (x)$ on $X$ with $1\leq |\EC (x)|\leq 2$ for all $x\in X$. 
For each $i\in \{ 1,2\}$ let $X_i$ be the set of all $x\in X$ with $|\EC (x)|=i$. 
Applying  \cref{retraction_via_1-end_selection} (to $F\rest{X_1}$ in place of $F$) and \cref{retraction_via_2-end_selection} (to $F\rest{X_2}$ in place of $F$) and discarding the $E_T$-saturation of an $F$-invariant Borel null set on which $F$ is smooth, we obtain a finite-to-one Borel retraction bundle to $F$ along $T$.
\cref{retraction=>end-embedding} therefore completes the proof.
\end{proof}

%% file: fig_lex-least_Voronoi.tex
\begin{tikzpicture}[scale=0.6, every node/.style={scale=0.9}]

\definecolor{myblue}{RGB}{30, 144, 255}

\draw[very thick] (0,0) -- (0,5.5);

\foreach \y [evaluate=\y as \n using int(\y+1)] in {1,2,3,4,5} {
\draw[very thick] (0,\y) -- (-\n,\y);

\foreach \x in {1,...,\y} {
\fill[black] (-\x,\y) circle (3.5pt);
}

\fill[red] (-\n,\y) circle (3.5pt);
\node[left, red] at (-\n-0.2, \y-0.03) {$\boldsymbol{x_\y}$};

\node[right] at (0.25,\y-0.03) {$\boldsymbol{y_\y}$};
\fill[black] (0,\y) circle (3.5pt);
}

\draw[very thick]
(0,0) to[in=100, out=-100] (-8,0);

\draw[very thick, dotted] (-8, 0) to[bend right=10] (-7.5, -1) node[font=\bfseries] 
at (-7.3, -1.4) {$\xi$};

\fill[red] (0,0) circle (3.5pt);
\node[red] at (0.5,-0.2) {$\boldsymbol{x}$};


\draw[line width=1.2pt, dotted]
(0, 5.8) to (0, 6.8);
\node[font=\bfseries] at (0, 7.3) {$\eta$};

\foreach \y in {5.8, 6.3, 6.8} {
\fill[black] (-3, \y) circle (1.5pt);
}

\foreach \i in {1,2,3} {
\pgfmathsetmacro{\x}{-6.2 - \i*0.5}
\pgfmathsetmacro{\y}{5.2 + \i*0.5}
\fill[red] (\x, \y) circle (1.5pt);
}


\draw[thick, gray, densely dashed, rounded corners=12pt] 
(-0.35, 6.0) -- (-0.35, 0.8) .. controls (-0.35, -0.5) and (0.35, -0.5) .. (0.35, 0.8) -- (0.35, 6.0);

\foreach \y [evaluate=\y as \n using int(\y+1)] in {1,2,3,4,5} {
\draw[thick, gray, densely dashed, rounded corners=6pt] 
(-0.7, \y-0.35) -- (-\n-0.3, \y-0.35) -- (-\n-0.3, \y+0.35) -- (-0.7, \y+0.35) -- cycle;
}

\end{tikzpicture}

%% file: fig_lex-least_retract.tex
\begin{tikzpicture}[scale=0.6, every node/.style={scale=0.9}]

\definecolor{myblue}{RGB}{30, 144, 255}

\draw[very thick]
(0,0) to[in=100, out=-100] (-8,0);

\draw[very thick, dotted] (-8, 0) to[bend right=10] (-7.5, -1) node[font=\bfseries] 
at (-7.3, -1.4) {$\xi$};

\foreach \y [evaluate=\y as \n using int(\y+1)] in {1,2,3,4,5} {

\draw[black, very thick] (0,0) to[bend right=30] (-\n,\y);

\fill[red] (-\n,\y) circle (3.5pt);
\node[left, red] at (-\n-0.03, \y-0.03) {$\boldsymbol{x_\y}$};

}

\fill[red] (0,0) circle (3.5pt);
\node[right, red] at (0,-0.1) {$\boldsymbol{x}$};

\foreach \y in {4.8, 5.3, 5.8} {
\fill[black] (-3, \y) circle (1.5pt);
}

\foreach \i in {1,2,3} {
\pgfmathsetmacro{\x}{-6.2 - \i*0.5}
\pgfmathsetmacro{\y}{5.2 + \i*0.5}
\fill[red] (\x, \y) circle (1.5pt);
}

\end{tikzpicture}

%% file: fig_contrast_with_lex_Voronoi.tex
\begin{tikzpicture}[scale=0.6, every node/.style={scale=0.9}]

\draw[very thick] (0,0) -- (0,5.5);

\definecolor{myblue}{RGB}{30, 144, 255}

\foreach \y [evaluate=\y as \n using int(\y+1)] in {1,2,3,4,5} {
\draw[very thick] (0,\y) -- (-\n,\y);

\draw[very thick]
(0,0) to[in=100, out=-100] (-8,0);

\draw[very thick, dotted] (-8, 0) to[bend right=10] (-7.5, -1) node[font=\bfseries] 
at (-7.3, -1.4) {$\xi$};

\foreach \x in {1,...,\y} {
\fill[black] (-\x,\y) circle (3.5pt);
}

\fill[red] (-\n,\y) circle (3.5pt);
\node[left, red] at (-\n-0.2, \y-0.03) {$\boldsymbol{x_\y}$};

\node[right] at (0.25,\y-0.03) {$\boldsymbol{y_\y}$};
\fill[black] (0,\y) circle (3.5pt);
}

\fill[red] (0,0) circle (3.5pt);
\node[red] at (0.5,-0.2) {$\boldsymbol{x}$};


\draw[line width=1.2pt, dotted]
(0, 5.8) to (0, 6.8);
\node[font=\bfseries] at (0, 7.3) {$\eta$};

\foreach \y in {5.8, 6.3, 6.8} {
\fill[black] (-3, \y) circle (1.5pt);
}

\foreach \i in {1,2,3} {
\pgfmathsetmacro{\x}{-6.2 - \i*0.5}
\pgfmathsetmacro{\y}{5.2 + \i*0.5}
\fill[red] (\x, \y) circle (1.5pt);
}


\draw[thick, gray, densely dashed, rounded corners=6pt] 
(-0.4, -1.1) -- (0.25, -0.4) -- (0.35, 0.35) -- (-0.35, 0.35) -- (-0.35, -0.2) -- (-0.8, -0.5);

\foreach \y [evaluate=\y as \n using int(\y+1)] in {1,2,3,4,5} {
\draw[thick, gray, densely dashed, rounded corners=6pt] 
(0.35, \y-0.35) -- (-\n-0.3, \y-0.35) -- (-\n-0.3, \y+0.35) -- (0.35, \y+0.35) -- cycle;
}

\end{tikzpicture}

%% file: fig_contrast_with_lex_retract.tex
\begin{tikzpicture}[scale=0.6, every node/.style={scale=0.9}]

\definecolor{myblue}{RGB}{30, 144, 255}

\draw[very thick]
(0,0) to[in=100, out=-100] (-8,0);

\draw[very thick, dotted] (-8, 0) to[bend right=10] (-7.5, -1) node[font=\bfseries] 
at (-7.3, -1.4) {$\xi$};

\draw[black, very thick] (0,0) to (-2,1);

\foreach \y [evaluate=\y as \n using int(\y+1)] in {1,2,3,4,5} {

\draw[black, very thick] (-\n,\y) to (-\n-1,\y+1);

\fill[red] (-\n,\y) circle (3.5pt);
\node[left, red] at (-\n-0.03, \y-0.03) {$\boldsymbol{x_\y}$};
}

\fill[red] (0,0) circle (3.5pt);
\node[right, red] at (0,-0.1) {$\boldsymbol{x}$};

\draw[very thick, dotted] (-7, 6) to (-7.9, 6.9) node[font=\bfseries] at (-8.1,7.2) {$\eta$};

\end{tikzpicture}

%% file: fig_parabolic_Voronoi.tex
\begin{tikzpicture}[scale=0.7, line join=round, line cap=round]

\definecolor{myblue}{RGB}{30, 144, 255}
\definecolor{myred}{RGB}{220, 20, 20}

\tikzstyle{bnode}=[circle, fill=black, inner sep=1.5pt]
\tikzstyle{rnode}=[circle, fill=myred, inner sep=1.5pt]


\coordinate (N0) at (0, 0);
\coordinate (N1) at (1.5, 0.6);
\coordinate (N2) at (2.6, 0.6);
\coordinate (N3) at (3.6, 0.2);
\coordinate (N4) at (4.5, -0.2);
\coordinate (N5) at (5.3, -0.5);
\coordinate (N6) at (6.3, -1.3);
\coordinate (N7) at (7.3, -1.3);

\coordinate (U1)  at (-0.4, 0.8);
\coordinate (D1)  at (-0.5, -0.9);
\coordinate (D1R) at (-1.3, -0.4);
\coordinate (D1B) at (-0.4, -1.9);

\coordinate (U2)  at (1.5, 1.7);
\coordinate (U2R) at (2.1, 2.6);

\coordinate (D2)  at (2.9, -0.4);
\coordinate (D2R) at (2.1, -1.2);
\coordinate (D2B) at (3.3, -1.1);

\coordinate (D4)  at (4.8, -1.2);
\coordinate (D4R) at (5.2, -2.0);

\coordinate (U5)   at (6.3, 0.3);
\coordinate (U5R1) at (6.3, 1.3);
\coordinate (U5B)  at (7.1, -0.1);
\coordinate (U5R2) at (7.7, 0.3);

\begin{scope}[densely dashed, thick, gray!80!black]
\draw plot [smooth cycle, tension=0.7] coordinates {(-0.1, 0.2) (1.5, 0.9) (2.6, 0.9) (3.7, 0.4) (3.7, -0.1) (2.6, 0.3) (1.5, 0.3) (-0.1, -0.4)};
\draw plot [smooth cycle, tension=0.7] coordinates {(-1.6, -0.25) (-1.2, -0.15) (-0.25, -0.8) (-0.5, -1.2) (-1.4, -0.7)};

\draw plot [smooth, tension=0.7] coordinates {(-0.65, -2.6) (-0.6, -1.75) (-0.2, -1.75) (-0.2, -2.6)};

\draw plot [smooth, tension=0.7] coordinates {(-0.85, 1.45) (-0.6, 0.6) (-0.1, 0.65) (-0.3, 1.6)};

\draw plot [smooth cycle, tension=0.7] coordinates {(1.25, 1.6) (2, 2.9) (2.4, 2.7) (1.7, 1.5)};
\draw plot [smooth cycle, tension=0.7] coordinates {(1.8, -1.2) (2.7, -0.1) (3.2, -0.4) (2.2, -1.5)};

\draw plot [smooth, tension=0.7] coordinates {(2.95, -2) (3.1, -0.95) (3.6, -0.9) (3.45, -2.1)};

\draw plot [smooth cycle, tension=0.7] coordinates {(4.3, 0.1) (4.7, 0.1) (4.9, -0.9) (5.5, -2.1) (5.2, -2.4) (4.6, -1.4) (4.3, -0.5)};
\draw plot [smooth cycle, tension=0.7] coordinates {(5.1, -0.4) (6.05, 0.5) (6.1, 1.5) (6.45, 1.5) (6.5, 0.2) (5.3, -0.75)};
\draw plot [smooth cycle, tension=0.7] coordinates {(6.9, 0) (7.7, 0.6) (7.9, 0.2) (7.1, -0.4)};
\draw plot [smooth cycle, tension=0.7] coordinates {(6.2, -1.05) (7.4, -1.05) (7.4, -1.55) (6.2, -1.55)};
\end{scope}

\begin{scope}[line width=1pt]
\draw (N0) to[bend right=10] (U1);
\draw (U1) -- (-0.6, 1.5);
\draw (U1) -- (-1.1, 1.0);
\draw (N0) to[bend left=10] (D1);
\draw (D1) -- (D1R);
\draw (D1R) -- (-1.5, 0.3);
\draw (D1R) -- (-2.1, -0.3);
\draw (D1) -- (D1B);
\draw (D1B) -- (-0.4, -2.6);

\draw (N1) to[bend right=5] (U2);
\draw (U2) -- (0.9, 2.4);
\draw (U2) -- (1.3, 2.7);
\draw (U2) -- (U2R);
\draw (U2R) -- (2.1, 3.4);
\draw (U2R) -- (2.7, 2.9);

\draw (N2) to[bend left=5] (D2);
\draw (D2) -- (D2R);
\draw (D2R) -- (1.8, -1.7);
\draw (D2) -- (D2B);
\draw (D2B) -- (3.7, -1.3);
\draw (D2B) -- (3.2, -2.0);

\draw (N4) to[bend left=5] (D4);
\draw (D4) -- (4.3, -1.8);
\draw (D4) -- (D4R);
\draw (D4R) -- (4.9, -2.8);

\draw (N5) to[bend left=5] (U5);
\draw (U5) -- (U5R1);
\draw (U5R1) -- (6.1, 2.1);
\draw (U5R1) -- (6.8, 1.9);
\draw (U5) -- (U5B);
\draw (U5B) -- (U5R2);
\draw (U5R2) -- (7.5, 0.9);
\draw (U5R2) -- (7.9, 0.9);
\draw (U5R2) -- (8.3, 0.3);
\end{scope}

\draw[color=myblue, line width=2pt] (N0) 
to[out=30, in=190] (N1) 
to[out=10, in=170] (N2) 
to[out=-10, in=165] (N3) 
to[out=-15, in=160] (N4) 
to[out=-20, in=165] (N5) 
to[out=-40, in=160] (N6) -- (N7);

\draw[color=myblue, line width=2pt] (N7) -- (7.8, -1.15);

\draw[color=myblue, line width=1.2pt, dotted] (7.8, -1.15) to[out=20, in=180] (8.6, -0.8) node[right, text=myblue, font=\bfseries] {$\xi_x$};


\node[rnode] at (N0) {};
\node[text=myred, below right=1pt, inner sep=0pt, font=\bfseries] at ($(N0) + (0.1, 0.1)$) {$x$};
\node[bnode] at (N1) {};
\node[bnode] at (N2) {};
\node[bnode] at (N3) {};
\node[bnode] at (N4) {};
\node[bnode] at (N5) {};
\node[rnode] at (N6) {};
\node[bnode] at (N7) {};

\node[bnode] at (U1) {};
\node[bnode] at (D1) {};
\node[rnode] at (D1R) {};
\node[bnode] at (D1B) {};

\node[bnode] at (U2) {};
\node[rnode] at (U2R) {};

\node[bnode] at (D2) {};
\node[rnode] at (D2R) {};
\node[bnode] at (D2B) {};

\node[bnode] at (D4) {};
\node[rnode] at (D4R) {};

\node[bnode] at (U5) {};
\node[rnode] at (U5R1) {};
\node[bnode] at (U5B) {};
\node[rnode] at (U5R2) {};

\end{tikzpicture}

%% file: fig_parabolic_retract.tex
\begin{tikzpicture}[scale=0.7, line join=round, line cap=round]

\definecolor{myblue}{RGB}{30, 144, 255}
\definecolor{myred}{RGB}{220, 20, 20}

\tikzstyle{bnode}=[circle, fill=black, inner sep=1.5pt]
\tikzstyle{rnode}=[circle, fill=myred, inner sep=1.5pt]


\coordinate (N0) at (0, 0);
\coordinate (N1) at (1.5, 0.6);
\coordinate (N2) at (2.6, 0.6);
\coordinate (N3) at (3.6, 0.2);
\coordinate (N4) at (4.5, -0.2);
\coordinate (N5) at (5.3, -0.5);
\coordinate (N6) at (6.3, -1.3);
\coordinate (N7) at (7.3, -1.3);

\coordinate (U1)  at (-0.4, 0.8);
\coordinate (D1)  at (-0.5, -0.9);
\coordinate (D1R) at (-1.3, -0.4);
\coordinate (D1B) at (-0.4, -1.9);

\coordinate (U2)  at (1.5, 1.7);
\coordinate (U2R) at (2.1, 2.6);

\coordinate (D2)  at (2.9, -0.4);
\coordinate (D2R) at (2.1, -1.2);
\coordinate (D2B) at (3.3, -1.1);

\coordinate (D4)  at (4.8, -1.2);
\coordinate (D4R) at (5.2, -2.0);

\coordinate (U5)   at (6.3, 0.3);
\coordinate (U5R1) at (6.3, 1.3);
\coordinate (U5B)  at (7.1, -0.1);
\coordinate (U5R2) at (7.7, 0.3);

\begin{scope}[line width=1pt]
\draw (D1R) to[bend left=8] (D1B);
\draw (D1R) to (N0);
\draw (N0) to (U1);
\draw (N0) to (U2R);
\draw (U2R) -- (2.1, 3.4);
\draw (U2R) -- (2.7, 2.9);
\draw (U2R) -- (1.7, 3.2);
\draw (U2R) -- (1.5, 2.7);
\draw (N0) to (D2R);
\draw (D2R) to ($(D2B)+(-0.4,-0.8)$);
\draw (D2R) -- (1.8, -1.7);
\draw (N0) to[bend left=20] (D4R);
\draw (D4R) to ($(D4R)+(-1,0.3)$);
\draw (D4R) to[bend left=5] ($(D4R)+(-1,0)$);
\draw (D4R) to[bend left=20] (U5R1);
\draw (U5R1) to[bend left=10] (N6);
\draw (U5R1) to (U5R2);
\draw (U5R1) -- (6.1, 2.1);
\draw (U5R1) -- (6.8, 1.9);
\draw (U5R2) -- (7.5, 0.9);
\draw (U5R2) -- (7.9, 0.9);
\draw (U5R2) -- (8.3, 0.3);
\draw (N6) to (N7);
\end{scope}

\node[rnode] at (N0) {};
\node[text=myred, below right=1pt, inner sep=0pt, font=\bfseries] at ($(N0) + (-0.2, -0.2)$) {$x$};
\node[rnode] at (N6) {};

\node[rnode] at (D1R) {};

\node[rnode] at (U2R) {};

\node[rnode] at (D2R) {};

\node[rnode] at (D4R) {};

\node[rnode] at (U5R1) {};
\node[rnode] at (U5R2) {};

\end{tikzpicture}

%% file: fig_loxodromic_Voronoi.tex
\begin{tikzpicture}[scale=0.7, line join=round, line cap=round]

\definecolor{myblue}{RGB}{30, 144, 255}
\definecolor{myred}{RGB}{220, 20, 20}

\tikzstyle{bnode}=[circle, fill=black, inner sep=1.5pt]
\tikzstyle{rnode}=[circle, fill=myred, inner sep=1.5pt]


\coordinate (N0) at (0, 0);
\coordinate (N1) at (1.5, 0.6);
\coordinate (N2) at (2.6, 0.6);
\coordinate (N3) at (3.6, 0.2);
\coordinate (N4) at (4.5, -0.2);
\coordinate (N5) at (5.3, -0.5);
\coordinate (N6) at (6.3, -1.3);
\coordinate (N7) at (7.3, -1.3);

\coordinate (U1)  at (-0.4, 0.8);
\coordinate (D1)  at (-0.5, -0.9);
\coordinate (D1R) at (-1.3, -0.4);
\coordinate (D1B) at (-0.4, -1.9);

\coordinate (U2)  at (1.5, 1.7);
\coordinate (U2R) at (2.1, 2.6);

\coordinate (D2)  at (2.9, -0.4);
\coordinate (D2R) at (2.1, -1.2);
\coordinate (D2B) at (3.3, -1.1);

\coordinate (D4)  at (4.8, -1.2);
\coordinate (D4R) at (5.2, -2.0);

\coordinate (U5)   at (6.3, 0.3);
\coordinate (U5R1) at (6.3, 1.3);
\coordinate (U5B)  at (7.1, -0.1);
\coordinate (U5R2) at (7.7, 0.3);

\begin{scope}[densely dashed, thick, gray!80!black]

\draw plot [smooth cycle, tension=0.7] coordinates {
(-0.6, -2.1)
(-0.8, -0.9)
(0.1, 0.4)
(1.6, 0.8)
(1.6, 0.4)
(0.2, -0.4)
(-0.2, -1.4)
(-0.1, -2.1)
};

\draw plot [smooth cycle, tension=0.7] coordinates {
(-1.3, -0.1) (-1.0, -0.4) (-1.3, -0.7) (-1.6, -0.4)
};

\draw plot [smooth, tension=0.7] coordinates {
(-0.3, 1.5) (-0.2, 0.7) (-0.5, 0.6) (-0.85, 1.4)
};

\draw plot [smooth cycle, tension=0.7] coordinates {
(1.2, 1.6) (1.9, 2.8) (2.4, 2.7) (1.7, 1.5)
};

\draw plot [smooth cycle, tension=0.7] coordinates {
(2.4, 0.9) (3.8, 0.4) (3.8, -0.1) (2.95, 0.2)  (3.1, -0.5) (2.3, -1.4) (1.8, -1.2) (2.6, -0.3)
};

\draw plot [smooth, tension=0.7] coordinates {
(3.5, -1.9) (3.6, -1) (3.1, -0.9) (2.95, -1.8) 
};

\draw plot [smooth cycle, tension=0.7] coordinates {
(4.3, 0) (4.65, 0) (5, -0.9) (5.5, -2.0) (5, -2.2) (4.5, -1.2)
};

\draw plot [smooth cycle, tension=0.7] coordinates {
(5.2, -0.15) (6.3, -1) (7.4, -1) (7.4, -1.6) (6.1, -1.6) (5.1, -0.7)
};

\draw plot [smooth cycle, tension=0.7] coordinates {
(6.05, 0.1) (6.05, 1.5) (6.55, 1.4) (6.55, 0.1)
};

\draw plot [smooth cycle, tension=0.7] coordinates {
(6.9, -0.4) (6.9, 0.1) (7.8, 0.6) (7.9, 0)
};

\end{scope}

\begin{scope}[line width=1pt]
\draw (N0) to[bend right=10] (U1);
\draw (U1) -- (-0.6, 1.5);
\draw (U1) -- (-1.1, 1.0);
\draw (D1) -- (D1R);
\draw (D1R) -- (-1.5, 0.3);
\draw (D1R) -- (-2.1, -0.3);

\draw (N1) to[bend right=5] (U2);
\draw (U2) -- (0.9, 2.4);
\draw (U2) -- (1.3, 2.7);
\draw (U2) -- (U2R);
\draw (U2R) -- (2.1, 3.4);
\draw (U2R) -- (2.7, 2.9);

\draw (N2) to[bend left=5] (D2);
\draw (D2) -- (D2R);
\draw (D2R) -- (1.8, -1.7);
\draw (D2) -- (D2B);
\draw (D2B) -- (3.7, -1.3);
\draw (D2B) -- (3.2, -2.0);

\draw (N4) to[bend left=5] (D4);
\draw (D4) -- (4.3, -1.8);
\draw (D4) -- (D4R);
\draw (D4R) -- (4.9, -2.8);

\draw (N5) to[bend left=5] (U5);
\draw (U5) -- (U5R1);
\draw (U5R1) -- (6.1, 2.1);
\draw (U5R1) -- (6.8, 1.9);
\draw (U5) -- (U5B);
\draw (U5B) -- (U5R2);
\draw (U5R2) -- (7.5, 0.9);
\draw (U5R2) -- (7.9, 0.9);
\draw (U5R2) -- (8.3, 0.3);
\end{scope}

\draw[color=myblue, line width=2pt] 
(-0.4, -2.6)
to[bend right=10] (D1B) 
to (D1) 
to (N0) 
to[out=30, in=190] (N1) 
to[out=10, in=170] (N2) 
to[out=-10, in=165] (N3) 
to[out=-15, in=160] (N4) 
to[out=-20, in=165] (N5) 
to[out=-40, in=160] (N6) -- (N7);

\draw[color=myblue, line width=2pt] (N7) -- (7.8, -1.15);

\draw[color=myblue, line width=1.2pt, dotted] (7.8, -1.15) to[out=20, in=180] (8.6, -0.8) node[right, text=myblue, font=\bfseries] {$\ell_x$};

\draw[color=myblue, line width=1.2pt, dotted] (-0.4, -2.6) to[bend left=20,out=30,in=200] (-1.1, -3.3) node[left, text=myblue, font=\bfseries] {$\ell_x$};


\node[rnode] at (N0) {};
\node[text=myred, below right=1pt, inner sep=0pt, font=\bfseries] at ($(N0) + (0.1, 0.1)$) {$x$};
\node[bnode] at (N1) {};
\node[bnode] at (N2) {};
\node[bnode] at (N3) {};
\node[bnode] at (N4) {};
\node[bnode] at (N5) {};
\node[rnode] at (N6) {};
\node[bnode] at (N7) {};

\node[bnode] at (U1) {};
\node[bnode] at (D1) {};
\node[rnode] at (D1R) {};
\node[bnode] at (D1B) {};

\node[bnode] at (U2) {};
\node[rnode] at (U2R) {};

\node[bnode] at (D2) {};
\node[rnode] at (D2R) {};
\node[bnode] at (D2B) {};

\node[bnode] at (D4) {};
\node[rnode] at (D4R) {};

\node[bnode] at (U5) {};
\node[rnode] at (U5R1) {};
\node[bnode] at (U5B) {};
\node[rnode] at (U5R2) {};

\end{tikzpicture}

%% file: fig_loxodromic_retract.tex
\begin{tikzpicture}[scale=0.7, line join=round, line cap=round]

\definecolor{myblue}{RGB}{30, 144, 255}
\definecolor{myred}{RGB}{220, 20, 20}

\tikzstyle{bnode}=[circle, fill=black, inner sep=1.5pt]
\tikzstyle{rnode}=[circle, fill=myred, inner sep=1.5pt]


\coordinate (N0) at (0, 0);
\coordinate (N1) at (1.5, 0.6);
\coordinate (N2) at (2.6, 0.6);
\coordinate (N3) at (3.6, 0.2);
\coordinate (N4) at (4.5, -0.2);
\coordinate (N5) at (5.3, -0.5);
\coordinate (N6) at (6.3, -1.3);
\coordinate (N7) at (7.3, -1.3);

\coordinate (U1)  at (-0.4, 0.8);
\coordinate (D1)  at (-0.5, -0.9);
\coordinate (D1R) at (-1.3, -0.4);
\coordinate (D1B) at (-0.4, -1.9);

\coordinate (U2)  at (1.5, 1.7);
\coordinate (U2R) at (2.1, 2.6);

\coordinate (D2)  at (2.9, -0.4);
\coordinate (D2R) at (2.1, -1.2);
\coordinate (D2B) at (3.3, -1.1);

\coordinate (D4)  at (4.8, -1.2);
\coordinate (D4R) at (5.2, -2.0);

\coordinate (U5)   at (6.3, 0.3);
\coordinate (U5R1) at (6.3, 1.3);
\coordinate (U5B)  at (7.1, -0.1);
\coordinate (U5R2) at (7.7, 0.3);

\begin{scope}[line width=1pt]
\draw (N0) to[bend left=8] (D1B);
\draw (D1R) to (N0);
\draw (N0) to (U1);
\draw (N0) to (U2R);
\draw (U2R) -- (2.1, 3.4);
\draw (U2R) -- (2.7, 2.9);
\draw (U2R) -- (1.7, 3.2);
\draw (U2R) -- (1.5, 2.7);
\draw (N0) to (D2R);
\draw (D2R) to ($(D2B)+(-0.4,-0.8)$);
\draw (D2R) -- (1.8, -1.7);
\draw (D2R) to[bend left=20] (D4R);
\draw (D4R) to ($(D4R)+(-1,0.3)$);
\draw (D4R) to[bend left=5] ($(D4R)+(-1,0)$);
\draw (D4R) to[bend left=20] (N6);
\draw (U5R1) to[bend left=10] (U5R1);
\draw (N6) to[bend left=5] (U5R1);
\draw (U5R1) to[bend right=5] (U5R2);
\draw (U5R1) -- (6.1, 2.1);
\draw (U5R1) -- (6.8, 1.9);
\draw (U5R2) -- (7.5, 0.9);
\draw (U5R2) -- (7.9, 0.9);
\draw (U5R2) -- (8.3, 0.3);
\draw (N6) to (N7);
\end{scope}

\node[rnode] at (N0) {};
\node[text=myred, below right=1pt, inner sep=0pt, font=\bfseries] at ($(N0) + (0, -0.2)$) {$x$};
\node[rnode] at (N6) {};

\node[rnode] at (D1R) {};

\node[rnode] at (U2R) {};

\node[rnode] at (D2R) {};

\node[rnode] at (D4R) {};

\node[rnode] at (U5R1) {};
\node[rnode] at (U5R2) {};

\end{tikzpicture}

%% file: amenable_subrelations.tex
\section{The structure of amenable equivalence relations in an acyclic graph}

In this section we describe the structure of amenable subrelations of treed mcp equivalence relations.

\subsection{Nonvanishing ends along subrelations}

To describe the structure of amenable subrelations of a treed mcp equivalence relation, we first record relativizations of our characterization of smoothness and amenability in terms of nonvanishing ends.

Recall the definition of an end of $T$ vanishing along a subset given in \cref{sec:retractions}.
Also recall that $\Xns_E$ is the $\mu$-essentially unique largest $E$-invariant Borel subset of $X$ on which $E$ is $\mu$-nowhere smooth. 
Note that if $F$ is a Borel subrelation of $E$, then we have $\Xns_{F} \subseteq \Xns_{E}$.

\begin{cor}\label{relative_trichotomy-1}
Let $T$ be an acyclic locally countable mcp Borel graph on a standard probability space $(X,\mu)$, let $Y$ be non-null Borel subset of $X$, and let $F$ be an amenable subrelation of $E_T \rest{Y}$.

\begin{enumerate}[(a)]
\item\label{relative_trichotomy:smooth} $F$ is smooth on a conull set if and only if for $\mu_Y$-a.e.\ $x \in Y$, every end $\eta$ of $T\rest{[x]_E}$ is vanishing along $[x]_F$.

\item\label{relative_trichotomy:amenable} $F$ is amenable and $\mu_Y$-nowhere smooth if and only if for $\mu_Y$-a.e.\ $x \in Y$, there are either one or two ends of $T$ which are nonvanishing along $[x]_F$.
\end{enumerate}
\end{cor}
\begin{proof}
\labelcref{relative_trichotomy:smooth} The forward implication follows immediately from \cref{smooth=rho-finite}.
For the converse, we repeat the same argument as in \cite[Proposition 3.3]{T-TD:mcp_Adams} for $F$, taking $C$ to vary over $F$-classes instead of $E$-classes, and replacing ``(non)vanishing'' with ``(non)vanishing along $C$.''

\labelcref{relative_trichotomy:amenable} The forward direction follows from \cref{end-preserving_slide} and \cref{topography:trichotomy}\labelcref{topography:trichotomy:amenable}.
For the backward direction, note that mapping $\mu_Y$-a.e.\ $x \in Y$ to the set $\EC_y$ of ends of $T$ which are nonvanishing along $[x]_F$ gives an $F$-invariant Borel end selection, witnessing the amenability of $F$ by \cref{end_selection}.
\end{proof}

\begin{remark}
After circulating an earlier draft of this article, it was pointed out to us by R.~Chen that in the proof of \cref{relative_trichotomy-1}, a simpler statement \cite[Proposition 2.5]{Chen-Kechris} can be used instead of the machinery of finite-to-one retractions (\cref{end-preserving_slide}).
However, \cref{end-preserving_slide} is needed below in \cref{parabolic_description,loxodromic_description}.
\end{remark}

\subsection{Parabolic and loxodromic structure}

Below we describe the parabolic and loxodromic parts of amenable subrelations in \cref{parabolic_description}  and \cref{loxodromic_description}, respectively.

\begin{theorem}[Parabolic structure]\label{parabolic_description}
Let $E$ be an mcp countable Borel equivalence relation on $(X,\mu)$ equipped with a treeing $T$, and let $\rho$ be the associated Radon--Nikodym cocycle. Assume that $E$ is $\mu$-nowhere smooth and let $F$ be an amenable subrelation of $E$. Then:

\begin{enumerate}[(i)]

\item \label{item:par:one_nonvanishing_end} For a.e.\ $x \in X^{\mathrm{par}}_F$, the end $\xi _F(x)$ is the only end of $T\rest{[x]_E}$ which is nonvanishing along $[x]_F$.
\end{enumerate}

\noindent Furthermore, there exists an acyclic countable-to-one Borel function $f : X^{\mathrm{par}}_F \to X^{\mathrm{par}}_F$ generating $F\rest{X^{\mathrm{par}}_F}$ satisfying the following:

\begin{enumerate}[resume*]
\item \label{item:par:end-preserving} The identity map on $[x]_F$ extends to a homeomorphism $\iota_{[x]_F} : \-{[x]_F}^T \to \-{[x]_F}^{T_f}$ for a.e.\ $x \in X^{\mathrm{par}}_F$.

\item\label{item:par:forward-end} $f^n(x)$ converges in $\-{[x]_F}^T$ to $\xi_F(x)$ for a.e.\ $x \in X^{\mathrm{par}}_F$.

\item\label{item:par:ess-1-ended} The unique maximal $F$-invariant Borel set $X_1\subseteq X^{\mathrm{par}}_F$ on which $T_f$ is essentially one-ended coincides a.e.\ with the set of all $x \in X^{\mathrm{par}}_F$ for which the back $f$-orbit of every point $y \in [x]_F$ is $\rho$-finite.

\item \label{item:par:nowhere-finitely-ended} The restriction of $T_f$ to $X^{\mathrm{par}}_F \setminus X_1$ is $\mu$-nowhere essentially finitely-ended, and for a.e.\ $x\in \Xpar_F\setminus X_1$ we have
\[
\limsup_{\substack{y\rightarrow \xi _F(x) \\ y\in [x]_F }}\rho ^x(y) = \limsup_{n \to \infty} \rho^x(f^n(x)) = \infty
\ \text{ and } \ 
\liminf_{\substack{y\rightarrow \xi _F(x) \\ y\in [x]_F }}\rho ^x(y) = 0.
\]
In addition, for every $\eta_f \in \del_{T_f} [x]_F$ distinct from the $\iota_{[x]_F}$-image of $\xi _F(x)$, the geodesic ray $[x,\eta_f)_{T_f}$, from $x$ to $\eta_f$ in $T_f$, is $\rho$-finite.
\end{enumerate}
\end{theorem}
\begin{proof}
Let $T_F$ be the treeing of $F$ given by \cref{end-preserving_slide}, so \labelcref{item:par:end-preserving} holds with $T_F$ in place of $T_f$. 
Let $f:X^{\mathrm{par}}_F\to X^{\mathrm{par}}_F$ be the function mapping each $x \in X^{\mathrm{par}}_F$ to its successor on the geodesic ray through $T_F$ from $x$ to the $\iota_{[x]_F}$-image of $\xi _F(x)$. 
Then $T_f = T_F$, hence \labelcref{item:par:end-preserving} and  \labelcref{item:par:forward-end} both hold.
Part \labelcref{item:par:ess-1-ended} follows from \cite[Theorem 5.3]{T-TD:mcp_Adams}. 
Part \labelcref{item:par:end-preserving} implies that $F\rest{X^{\mathrm{par}}_F}$ is parabolic in $T_f$ and hence $T_f \rest{X^{\mathrm{par}}_F \setminus X_1}$ is $\mu$-nowhere essentially finitely-ended. 
Part \labelcref{item:par:one_nonvanishing_end} and the rest of part \labelcref{item:par:nowhere-finitely-ended} are given by \cite[Proposition 7.5]{T-TD:mcp_Adams}.
\end{proof}

\begin{theorem}[Loxodromic structure]\label{loxodromic_description}
Let $E$ be an mcp countable Borel equivalence relation on $(X,\mu)$ equipped with a treeing $T$, and let $\rho$ be the associated Radon--Nikodym cocycle. Assume that $E$ is $\mu$-nowhere smooth and let $F$ be an amenable subequivalence relation of $E$. Then:

\begin{enumerate}[(i)]
\item \label{item:lox:two_nonvanishing_ends} For a.e.\ $x \in \Xlox_F$, the two ends in $\EC _F(x)$ are the only ends of $T\rest{[x]_E}$ which are nonvanishing along $[x]_F$.
\end{enumerate}

\noindent Furthermore, there is an essentially two-ended treeing $T_F$ of $F\rest{\Xlox_F}$ satisfying the following:

\begin{enumerate}[resume*]
\item \label{item:lox:end-preserving} The identity map on $[x]_F$ extends to a homeomorphism $\iota_{[x]_F} : \-{[x]_F}^T \to \-{[x]_F}^{T_F}$ for a.e.\ $x \in \Xlox_F$.

\item\label{item:lox:the-two-ends} There is a unique (up to a null set) $T_F$-convex Borel $F\rest{\Xlox_F}$-complete section $L$ such that $T_F \rest{L}$ is two-regular. For a.e.\ $x \in L$, the two ends of the $T_F \rest{L}$-connected component of $x$ are exactly the $\iota_{[x]_F}$-images of the two ends in $\EC_F(x)$.

\item \label{item:lox:cocycle_behavior} For a.e.\ $x \in \Xlox_F$, letting $\xi,\eta \in \EC_F(x)$ be the two distinct ends of $T$ selected by $x$ and letting $\xi_F,\eta_F$ be their $\iota_{[x]_F}$-images, we have
\begin{align*}
\liminf_{\substack{y\rightarrow \xi \\ y \in [x]_F}} \rho^x(y) = \liminf_{\substack{y\rightarrow \eta \\ y \in [x]_F}}\rho^x(y) < \infty
    \ &\text{ and } \ 
\limsup_{\substack{y\rightarrow \xi \\ y \in [x]_F}} \rho ^x(y) = \limsup _{\substack{y\rightarrow \eta \\ y \in [x]_F}}\rho ^x(y) > 0,
\\
\liminf_{y \in [x,\xi_F)_{T_F}} \rho^x(y) = \liminf_{y \in [x,\eta_F)_{T_F}} \rho^x(y) < \infty
    \ &\text{ and } \ \limsup_{y \in [x,\xi_F)_{T_F}} \rho ^x(y) = \limsup_{y \in [x,\eta_F)_{T_F}} \rho^x(y) > 0.
\end{align*}
In addition, if we let $X'$ denote the set of all $x \in \Xlox_F$ for which one of these limits inferior is positive or one of these limits superior is finite, then the restriction of $\rho$ to $F \rest{X'}$ is a coboundary.
\end{enumerate}
\end{theorem}
\begin{proof}
Let $T_F$ be the treeing of $F\rest{\Xlox_F}$ given by \cref{end-preserving_slide}, so \labelcref{item:lox:end-preserving} holds. The existence of a set $L$ as in \labelcref{item:lox:the-two-ends} follows from the definition of being essentially two-ended, and the uniqueness follows from the definition of $\Xlox_F$. Part \labelcref{item:lox:end-preserving} implies that $F\rest{\Xlox_F}$ is loxodromic in $T_F$, which implies the rest of part \labelcref{item:lox:the-two-ends}. 

Towards part \labelcref{item:lox:two_nonvanishing_ends}, for a.e.\ $x \in \Xlox_F$, by \cref{relative_trichotomy-1}\labelcref{relative_trichotomy:amenable}, there are exactly one or two ends of $T$ which are nonvanishing along $[x]_F$, so the $\mu$-essential maximality of $\EC_F$ implies all ends not in $\EC_F(x)$ are vanishing along $[x]_F$.
That both ends in $\EC_F(x)$ are nonvanishing along $[x]_F$ will follow once part \labelcref{item:lox:cocycle_behavior} is established. 

Part \labelcref{item:lox:cocycle_behavior} would follow from part \labelcref{item:lox:end-preserving} and \cite[Proposition 7.4]{T-TD:mcp_Adams} if $T_F$ were generated by a function, which may not be the case. 
A minor modification of the proof of \cite[Proposition 7.4]{T-TD:mcp_Adams} applied to $T_F$ could be used to obtain \labelcref{item:lox:cocycle_behavior}, but instead we will establish \labelcref{item:lox:cocycle_behavior} by directly using \cite[Proposition 7.4]{T-TD:mcp_Adams} together with an end-disentanglement argument.

Consider the space
\[
\tilde{X} \defeq \{(x,\xi) \, : \, x \in \Xlox_F \text{ and } \xi \in \iota_{[x]_F}(\EC_F(x))\} ,
\]
equipped with the sigma algebra coming from its identification with a Borel subset of $\Rays(T_F)$,
and with the finite Borel measure $\tilde{\mu}$ defined by $\tilde{\mu}(A) \defeq \frac{1}{2} \int _{\Xlox_F} |A_x| \, d\mu$. 
Let $\tilde{f} : \tilde{X} \to \tilde{X}$ send $(x,\xi)$ to $(y,\xi)$, where $y$ is the successor of $x$ on the geodesic ray through $T_F$ from $x$ to $\xi$. 
Note that $\tilde{L} \defeq \set{(x,\xi) \in \tilde{X} : x \in L}$ is a $T_{\tilde{f}}$-convex Borel $E_{\tilde{f}}$-complete section and $\tilde{f}$ restricts to a bijection of $\tilde{L}$ with itself, so $\tilde{f}$ is essentially two-ended.
Next, observe that the Radon--Nikodym cocycle $\tilde{\rho}$, of $E_{\tilde{f}}$ with respect to $\tilde{\mu}$, is given by $\tilde{\rho}((x,\xi), (y, \xi)) = \rho(x,y)$, so applying \cite[Proposition 7.4]{T-TD:mcp_Adams} to $\tilde{f}$ indeed gives \labelcref{item:lox:cocycle_behavior}.
\end{proof}

The following corollaries are immediate consequences of \cref{parabolic_description} and \cref{loxodromic_description}.

\begin{cor}[Topography selects the ends]\label{topography_selects_ends}
Let $E$ be an mcp countable Borel equivalence relation on $(X,\mu)$ equipped with a treeing $T$, and let $F$ be an amenable subrelation of $E$.
Then for a.e.\ $x \in \Xns_F$, the ends in $\EC_F(x)$ are precisely the ends of $T \rest{[x]_E}$ which are nonvanishing along $[x]_F$.

Moreover, if there is an $F$-invariant probability measure equivalent to $\mu$, then $\EC_F(x)$ are exactly the ends of $T \rest{\Conv_T ([x]_F)}$ for a.e.\ $x \in \Xns_F$.
\end{cor}
\begin{proof}
The main statement follows from \cref{parabolic_description}\labelcref{item:par:one_nonvanishing_end} and \cref{loxodromic_description}\labelcref{item:lox:two_nonvanishing_ends}.
The moreover statement follows from \cref{end-preserving_slide} and \cite[Theorem 9.8]{T-TD:mcp_Adams}.
\end{proof}

In contrast to the pmp setting, in the general mcp setting the $T$-convex hull of the $F$-classes may not be so closely related to the maximal $F$-invariant end selection.
We give an example where the $T$-convex hull of each $F$-class contains every end of the entire $T$-component of that class.

\begin{example}\label{example:full_convex_hull}
Let $\F_2 \actson (X,\mu)$ be a free pmp action of the free group $\F_2 \defeq \gen{a,b}$, so $E^X_{\F_2}$ is treed by the Schreier graphing $T_0$ of this action.
Let $E^X_a$ be the orbit equivalence relation of the action of the generator $a$, so it is a maximal amenable subrelation of $E^X_{\F_2}$ (any larger subrelation is treeable of cost strictly larger than $1$, hence is nonamenable).
Let $p : X \times \F_2 \to X$ be the projection onto $X$ and let $E$ and $F$ be the pull-backs of $E^X_{\F_2}$ and $E^X_a$ to $X \times \F_2$ via $p$.
We identify $X$ with $X \times \set{e_{\F_2}}$ and extend the treeing $T_0$ to the treeing $T$ of $E$ defined by adding an edge between $(x,\gamma)$ and $(\gamma^{-1} . x, e_{\F_2})$ for each $x \in X$ and $\gamma \in \F_2$.
It is now easy to observe that the $T$-convex hull of each $F$-class contains every end of the entire $T$-component of that class.
Although $T$ is locally infinite, a cosmetic modification of this example as in the proof of \cite[Proposition 3.10]{JKL} -- by additionally attaching an $\N$-indexed ray to each vertex of $T_0$ and sliding the edges of $T \setminus T_0$ up -- would yield a bounded degree treeing with the same property.
\end{example}

\begin{cor}[Permanence of end selection]\label{parabolic_in_parabolic}
Let $E$ be an mcp countable Borel equivalence relation on $(X,\mu)$ equipped with a treeing $T$, and let $\rho$ be the associated Radon--Nikodym cocycle. 
Let $F_0 \subseteq F_1$ be amenable subequivalence relations of $E$.
Then $\EC_{F_0}(x) =\EC_{F_1}(x)$ for a.e.\ $x\in X_{F_0}^{\mathrm{ns}}$.

Hence, after discarding a null set we have $X_{F_0}^{\mathrm{par}}\subseteq X_{F_1}^{\mathrm{par}}$ and $X_{F_0}^{\mathrm{lox}}\subseteq X_{F_1}^{\mathrm{lox}}$.
In particular, if $F_0$ is $\mu$-nowhere smooth and $F_1$ is parabolic in $T$, then $F_0$ is parabolic in $T$ as well.
\end{cor}

\begin{proof}
Since $\Xns_{F_0} \subseteq \Xns_{F_1}$, it follows immediately from the maximality of the end selection $\EC_{F_0}$ that $\Xpar_{F_0} \subseteq \Xpar_{F_1}$, so it is enough to show that the set $\Xlox_{F_0} \subseteq \Xlox_{F_1}$ after discarding a null set, but this follows from \cref{topography_selects_ends}.
\end{proof}

\subsection{Relative uniqueness and almost malnormality of maximal amenable subrelations}

\begin{theorem}[Relative uniqueness]\label{unique_max_amenable}
Let $E$ be an mcp treeable countable Borel equivalence relation on $(X,\mu)$ and let $F$ be an amenable subequivalence relation of $E$. 
Then $F \rest{\Xns_F}$ is contained in a ($\mu$-essentially) unique maximal amenable subequivalence relation $\hat{F}$ of $E \rest{\Xns_F}$. 
In particular, every $\mu$-nowhere smooth amenable subequivalence relation of $E$ is contained in a ($\mu$-essentially) unique maximal amenable subequivalence relation of $E$.

Moreover, given any treeing $T$ of $E$, we have that
\begin{equation}
\label{eq:maximal} \hat{F} = \set{ (x,y) \in E \rest{\Xns_F} : \EC_F(x) = \EC_F(y) }.
\end{equation}
\end{theorem}

\begin{proof}
Fix a treeing $T$ of $E$ and let $\hat{F}$ be given by \eqref{eq:maximal}. Then $\hat{F}$ is an amenable subequivalence relation of $E \rest{\Xns_F}$ containing $F \rest{\Xns_F}$. If $\tilde{F}$ is another amenable subequivalence relation of $E \rest{\Xns_F}$ containing $F \rest{\Xns_F}$, then $\EC_F = \EC_{\tilde{F}}$ by \cref{parabolic_in_parabolic}, and hence $\tilde{F} \subseteq \hat{F}$.
\end{proof}

\begin{cor}\label{nowhere_smooth_intersection}
Let $E$ be an mcp treeable countable Borel equivalence relation on $(X,\mu)$.
Let $F_0$ and $F_1$ be amenable subequivalence relations of $E$, and let $Y \defeq \Xns_{F_0 \cap F_1}$. 
Then the equivalence relation generated by $F_0 \rest{Y}$ and $F_1 \rest{Y}$ together is still amenable. 
In particular, if $F_0 \cap F_1$ is $\mu$-nowhere smooth, then $F_0$ and $F_1$ together generate an amenable equivalence relation.
\end{cor}

\begin{proof}
Because treeability is closed under restrictions, we may assume that $F_0 \cap F_1$ is $\mu$-nowhere smooth and hence $Y = X$. Let $\hat{F}_0$ and $\hat{F}_1$ be maximal amenable subequivalence relations of $E$ containing $F_0$ and $F_1$, respectively. Then $\hat{F}_0 = \hat{F}_1$ by the previous theorem applied to $F_0 \cap F_1$.
\end{proof}

\begin{cor}[Almost malnormality]\label{malnormality}
Let $E$ be an mcp treeable countable Borel equivalence relation on $(X,\mu)$ and let $F$ be a maximal amenable subequivalence relation of $E$.
Then for every $\phi \in [[E]]$, after discarding a null set, there is a partition of $\mathrm{dom}(\phi )$ into two Borel sets $A$ and $B$ such that $\phi\rest{A} \in [[F]]$ and $F\rest{\phi (B)}\cap (\phi \times \phi ) (F\rest{B})$ is smooth.
\end{cor}

\begin{proof}
We may assume without loss of generality that $E$ is $\mu$-nowhere smooth, and therefore $F$ is $\mu$-nowhere smooth as well, by maximality.
Let $D \defeq \dom (\phi )$, let $F_0\defeq  F\rest{\phi (D)}\cap (\phi \times \phi ) (F\rest{D})$, let $A \defeq \phi ^{-1}(\Xns_{F_0})$, and let $B\defeq D\setminus A$.
Since $F\rest{\phi (A)}$ and $(\phi \times \phi ) (F\rest{A})$ are both maximal amenable subrelations of $E\rest{\phi (A)}$ containing the $\mu\rest{\phi (A)}$-nowhere smooth subrelation $F_0\rest{\phi (A)}$, we deduce from \cref{unique_max_amenable} that $(\phi \times \phi ) (F\rest{A})=F\rest{\phi (A)}$, i.e., $\phi\rest{A}$ normalizes $F$.

Fix a treeing $T$ of $E$. 
Since $\phi\rest{A}$ normalizes $F$, the assignments $x\mapsto \EC_{F}(\phi (x))\cup \EC_{F}(x)$ and $x\mapsto \EC_{F}(\phi ^{-1}(x))\cup \EC_F(x)$ are $F$-invariant Borel $T$-end selections on $A$ and $\phi (A)$ respectively. 
Maximality of $\mathcal{E}_F$ therefore implies that, after discarding a null set, for each $x\in A$ we have
\[
\EC_F(x)= \EC_F(\phi ^{-1}(\phi (x)))\subseteq \EC_F(\phi (x))\subseteq \EC_F(x),
\]
hence $\EC_F(\phi (x))=\EC_F(x)$.
Therefore $\phi\rest{A} \in [[F]]$ by $\mu$-nowhere smoothness and maximality of $F$ together with \labelcref{eq:maximal}.
\end{proof}

The results in this subsection are consistent with a positive answer to the following important open question, first communicated to the authors by Damien Gaboriau (in the pmp setting):

\begin{question}\label{q:max_amen_implies_free_factor?}
Let $E$ be a treeable mcp countable Borel equivalence relation.
Is every maximal amenable subequivalence relation of $E$ a free factor of $E$?
\end{question}

\subsection{Transversality of classes and subrelations}

\begin{theorem}[Transversality of classes]\label{transversality}
Let $E$ be an mcp countable Borel equivalence relation on $(X,\mu )$ that is $\mu$-nowhere smooth, and let $T$ be a treeing of $E$. 
Let $F$ be a maximal amenable subequivalence relation of $E$. 
Then, after discarding a null set, the sets $\EC _F (x) $ and $\EC _F (y)$ are disjoint for all $F$-inequivalent points $x,y\in X$.
\end{theorem}

\begin{proof}
We use an end-disentanglement argument, as in the proof of \cref{loxodromic_description}. Consider the space 
\[
\tilde{X} = \{ (x,\xi ) : x\in X \text{ and }\xi \in \EC _F (x) \} ,
\]
equipped with the finite measure $\tilde{\mu}$ defined by $\tilde{\mu}(A) \defeq \int _X |A_x| \, d\mu$. 
Let $p:\tilde{X}\rightarrow X$ be the natural projection, and define $\tilde{E}\defeq (p\times p ) ^{-1}(E)$, and $\tilde{F}\defeq (p\times p)^{-1}(F)$. Let $x\mapsto \xi _x \in \EC _F (x)$ be a (not necessarily $F$-invariant) Borel selection of one end from $\EC _F (x)$, and note that the set $\tilde{X}_0\defeq \{ (x,\xi _x):x\in X\}$ is an $\tilde{F}$-complete section on which $\tilde{E}$ and $\tilde{F}$ are isomorphic to $E$ and $F$ respectively. 
In particular, $\tilde{F}$ is a maximal amenable subequivalence relation of $\tilde{E}$.

Note that $F$ is $\mu$-nowhere smooth, being a maximal amenable subequivalence relation of the $\mu$-nowhere smooth equivalence relation $E$. It follows that $\tilde{F}$ and $\tilde{E}$ are each $\tilde{\mu}$-nowhere smooth. 

Define a graph $\tilde{T}$ on $\tilde{X}$ by
\[
\tilde{T} = \{ ((x,\xi _x),(y,\xi _y)) \, : \, (x,y)\in T \} \cup \{ ((x,\xi ), (x,\eta )) \in \tilde{X}  \, : \, x\in X \text{ and }\xi\neq \eta \} .
\]
Then $\tilde{T}$ is a treeing of $\tilde{E}$, and the projection $p$ extends continuously to a bijective map $\del p : \del_{\tilde{T}} \tilde{X} \to \del_T X$ which identifies the ends of $\tilde{T}$ with the ends of $T$. 
Then the maximal $\tilde{F}$-invariant end selection with respect to $\tilde{T}$ is given by $\EC _{\tilde{F}}(x,\xi )= (\del p )^{-1}(\EC _F (x))$.  

Let $\tilde{Q}$ be the subequivalence relation of $\tilde{E}$ defined by
\[
\tilde{Q} \defeq \{ ((x,\xi ), (y,\eta ))\in \tilde{E} : \xi = \eta \} .
\]
Observe that $\tilde{Q}$ is amenable by \cref{end_selection}, since there is a $\tilde{Q}$-invariant Borel end selection of exactly one end out of each $\tilde{T}$-component, namely, $(x,\xi) \mapsto (\del p)^{-1}(\xi)$.
Furthermore, each $\tilde{F}$-class contains at most two $\tilde{Q}\cap\tilde{F}$-classes, hence $\tilde{Q} \cap \tilde{F}$ is $\tilde{\mu}$-nowhere smooth. 
In addition, since $\tilde{E}$ is treeable and $\tilde{Q}$ and $\tilde{F}$ are amenable subequivalence relations whose intersection is $\tilde{\mu}$-nowhere smooth, it follows from \cref{nowhere_smooth_intersection} that the equivalence relation generated by $\tilde{Q}$ and $\tilde{F}$ is amenable. 
Since $\tilde{F}$ is maximal amenable, after discarding a null set we have $\tilde{Q}\subseteq \tilde{F}$. Therefore, if $x,y\in X$ are such that $\EC _F(x)$ and $\EC _F(y)$ are not disjoint, say $\xi \in \EC _F(x)\cap \EC _F(y)$, then $((x,\xi),(y,\xi) ) \in \tilde{Q}$, so $((x,\xi ),(y,\xi )) \in \tilde{F}$, and hence $(x,y)\in F$. 
\end{proof}

The following statement in the pmp setting again contrasts with \cref{example:full_convex_hull}.

\begin{cor}[Convex hulls in the pmp setting]\label{convex_hulls}
Let $E$ be a pmp countable Borel equivalence relation on $(X,\mu)$ that is $\mu$-nowhere smooth, and let $T$ be a treeing of $E$. 
Let $F$ be a maximal amenable subequivalence relation of $E$. 
Then, after discarding a null set, we have:

\begin{enumerate}[(a),leftmargin=*]
\item\label{convex_hull_loc_fin} $T \rest{\Conv_T ([x]_F)}$ is locally finite for each $x \in X$.

\item\label{finite_intersection} The intersection $\Conv_T ([x]_F)  \cap \Conv_T ([y]_F)$ is finite for all $F$-inequivalent points $x,y\in X$.
\end{enumerate}
\end{cor}
\begin{proof}
Recall that treeings of pmp amenable equivalence relations are locally finite, in fact, expected degree at most $2$ \cite{Levitt}.
Using this, part \labelcref{convex_hull_loc_fin} follows from \cref{retraction_via_1-end_selection,retraction_via_2-end_selection} since the obtained retractions are finite-to-one.

Part \labelcref{finite_intersection} now follows from \labelcref{convex_hull_loc_fin} together with \cref{topography_selects_ends,transversality}.
\end{proof}

\begin{cor}[Transversality of subrelations]\label{relation_transversality}
Let $E$ be an mcp countable Borel equivalence relation on $(X,\mu )$ that is $\mu$-nowhere smooth, and let $T$ be a treeing of $E$. Let $F_0$ and $F_1$ be maximal amenable subequivalence relations of $E$. 
Then, after discarding a null set, for each $x\in X$ the sets $\EC_{F_0}(x)$ and $\EC_{F_1}(x)$ are either equal or disjoint.
In particular, $F_0$ and $F_1$ coincide on the set $A=\{ x\in X : \EC_{F_0}(x)\cap \EC_{F_1}(x)\neq\emptyset \}$.
\end{cor}

\begin{proof}
It suffices to prove that $F_0\rest{A}=F_1\rest{A}$.
Let $x\mapsto \xi _x$ be a Borel map on $A$ with $\xi _x \in \EC_{F_0}(x)\cap\EC_{F_1}(x)$, and let $R$ be the subequivalence relation of $E\rest{A}$ given by
\[
R=\{ (x,y)\in E\rest{A} : \xi _x =\xi _y \} .
\]
By \cref{transversality} we may assume, after discarding a null set, that for each $i\in \{ 0, 1\}$ the sets $\EC _{F_i}(x)$ and $\EC_{F_i}(y)$ are disjoint for all $F_i$-inequivalent points $x,y\in X$.  
Then $R\subseteq (F_0\rest{A})\cap (F_1\rest{A})$ since for $(x,y)\in R$ and $i\in \{ 0, 1\}$ the sets $\EC _{F_i}(x)$ and $\EC _{F_i}(y)$ both contain $\xi _x=\xi _y$, hence are not disjoint.
Moreover, each $F_0\rest{A}$-class contains at most two $R$-classes, since $|\EC_{F_0}(x)|\leq 2$ for all $x\in X$. 
It follows that $R$ is nowhere smooth.
Since $F_0\rest{A}$ and $F_1\rest{A}$ are both maximal amenable in $E\rest{A}$, and their intersection is nowhere smooth, it follows from \cref{nowhere_smooth_intersection} that $F_0\rest{A}=F_1\rest{A}$.
\end{proof}

\subsection{Conjugacy}

Let $E$ be an mcp countable Borel equivalence relation on a standard probability space $(X,\mu)$.
Recall that $[[E]]$ denotes the semigroup of all partial injections whose graph is Borel and contained in $E$. 
We let $[E]$ denote the \dfn{Borel full group of $E$}, i.e., the group  of all Borel automorphisms of $X$ whose graph is contained in $E$.

For Borel subrelations $R_0$ and $R_1$ of $E$, we say that $R_0$ and $R_1$ are \dfn{conjugate within $E$} if, after discarding a null set, $R_1 = (\phi \times \phi)(R_0)$ for some $\phi \in [E]$.

Suppose that $A_0$ and $A_1$ are Borel subsets of $X$, and $F_0$ and $F_1$ are subequivalence relations of $E|_{A_0}$ and $E|_{A_1}$ respectively.
We say that $F_0$ and $F_1$ are \dfn{complete section conjugate within $E$} if, after discarding a $\mu$-null set, 
for each $i\in\{ 0,1 \}$ there exist Borel $F_i$-complete sections $C_i\subseteq A_i$, and some $\phi \in [[E]]$ with $C_0\subseteq \dom (\phi )$ and $C_1 \subseteq \im(\phi )$ such that $(\phi \times \phi )(F_0|_{C_0})=(F_1|_{C_1})$. 

\begin{cor}[Transversality and conjugacy]\label{Transversality_and_conjugacy}
Let $E$ be an mcp countable Borel equivalence relation on $(X,\mu )$ that is $\mu$-nowhere smooth, and assume that $T$ is a treeing of $E$. 
Let $F_0$ and $F_1$ be maximal amenable subrelations of $E$, and
consider the bipartite graph $G$ on the disjoint union $X/F_0 \sqcup X/F_1$ given by
\[
G = \{ ([x]_{F_0},[y]_{F_1}) : (x,y)\in E \text{ and } \EC_{F_0}(x)\cap\EC_{F_1}(y)\neq\emptyset\} .
\]
Then, after discarding a null set, each vertex of $G$ has degree at most $1$ (i.e., $G$ is a partial matching), $\EC_{F_0}(x)=\EC_{F_1}(y)$ for all $([x]_{F_0},[y]_{F_1})\in G$, and the restrictions $F_0\rest{A_0}$ and $F_1\rest{A_1}$ are complete section conjugate within $E$, where
\[
A_i\defeq \{ x\in X : \mathrm{deg}_G([x]_{F_i}) = 1\} .
\]
In particular, if $F_0$ and $F_1$ are ergodic and are not complete section conjugate within $E$ then, after discarding a null set, the sets $\EC_{F_0}(x)$ and $\EC_{F_1}(y)$ are disjoint for all $(x,y)\in E$. 
\end{cor}

\begin{proof}
By \cref{transversality} and the fact that $\EC_{F_i}(x) \le 2$ for $i=0,1$, the degree of each vertex of $G$ is at most $2$ and we need to show that it is at most $1$ a.e.
For each $i=0,1$, let $\tilde{F}_i$ be the equivalence relation on $X$ defined by $(x_0,x_1) \in \tilde{F}_i$ if and only if the $F_i$-classes of $x_0$ and $x_1$ are in the same $G$-component.
By definition, $\tilde{F}_i \supseteq F_i$ and we define a graph $\tilde{G}_i$ as the restriction to $X/F_i$ of the distance-2 graph $G^2$ of $G$.
Then each $\tilde{G}_i$ has degree at most $2$, which implies that $F_i$ is Borel 1-coamenable in $\tilde{F}_i$.
Hence $\tilde{F}_i$ is Borel 2-amenable by \cref{amenable-coamenable}, and thus $\mu$-amenable by \cref{2-amen_implies_mu-amen}.
By maximality, after discarding a null set, $\tilde{F}_i = F_i$, so $G$ has degree at most $1$.
The fact that $F_0 \rest{A_0}$ and $F_1 \rest{A_1}$ are conjugate when restricted to some Borel complete sections now follows from Luzin--Novikov uniformization.
\end{proof}

\begin{example}
In general, complete section conjugacy does not imply conjugacy, even for ergodic maximal amenable subrelations of mcp treeable equivalence relations.
For example, let $\Gamma$ be the free group on two generators $a$ and $b$, and take a free mcp action of $\Gamma$ on $(X,\mu )$ where the action of the subgroup $\langle a \rangle$ is measure preserving and ergodic, and the action of the subgroup $\langle b\rangle$ is ergodic and mcp, and such that the associated Radon--Nikodym cocycle $\rho ^x(bx)$ is equal to a constant $c>1$ for all $x$ in a positive measure subset $C$ of $X$.
For example, take $(X,\mu ) \defeq ([0,1]^\Gamma \times 2^{\N}, \lambda ^{\Gamma}\times \nu_{2/3}^\N )$ (where $\nu_{2/3} (\set{1}) \defeq 2/3$) and the product action determined by $a.(z,y) \defeq (\sigma _a(z),\tau _a(y))$ and $b.(z,y) \defeq (\sigma _b(z),\tau _b(y))$, 
where $\sigma$ is the Bernoulli shift action of $\Gamma$, and $\tau _a$ is an arbitrary invertible ergodic pmp transformation on $(2^\N, \nu_{2/3}^\N)$, and $\tau _b$ is the odometer on $2 ^\N$. 

Let $E$ be the orbit equivalence relation of the action of $\Gamma$. 
Since $\langle a\rangle$ is ergodic, after discarding a null set the sets $C$ and $b.C$ are both $E_{\langle a\rangle}$-complete sections.
We may then find a Borel retraction $\psi : X\to b.C$ from $X$ to $b.C$ with $(x,\psi x)\in E_{\langle a\rangle}$ and hence $\rho ^x (\psi x)= 1$ for all $x\in X$. 
Let 
\[
F_0 \defeq E_{\langle a\rangle} \text{ and } F_1 \defeq (\psi \times \psi )^{-1}(E_{\langle bab^{-1}\rangle}|_{b.C} ),
\]
and let $T$ be the Schreier graph of the action of $\Gamma$ (with respect to the generators $a$ and $b$), which is a treeing of $E$.

Both $F_0$ and $F_1$ are amenable and $\mu$-nowhere smooth, and we claim they are maximal amenable in $E$: for $F_0$ this follows from the characterization \labelcref{eq:maximal}, since $\EC_{F_0}(x)$ consists of the two ends spanned by the bi-infinite line $\langle a\rangle . x = [x]_{F_0}$ through $T$ (and distinct $F_0$-classes give distinct such lines); for $F_1$, we have $\EC_{F_1}(x)=\EC_{F_0}(b^{-1}.\psi x)$ since the map $x\mapsto \EC_{F_0}(b^{-1}.\psi x)$ is an $F_1$-invariant selection of exactly two ends, so \labelcref{eq:maximal} applies again to establish maximality of $F_1$. 
The relations $F_0$ and $F_1$ are complete section conjugate in $E$ since $(b\times b).(F_0|_C)= F_1|_{b.C}$.

We show that $F_0$ and $F_1$ are not conjugate in $E$. 
Suppose toward a contradiction that $\phi \in [E]$ is such that $(\phi \times \phi )(F_0) =F_1$. 
Then, after discarding a null set, every $x\in X$ must satisfy $\EC _{F_1}(\phi (x))=\EC _{F_0}(x)$. 
By maximality of $F_1$, for $x\in C$, the only $F_1$-class in $[x]_E$ that selects the same ends of $T$ as $[x]_{F_0}$ is $[b.x]_{F_1}$, and hence $\phi ([x]_{F_0})= [b.x]_{F_1}$. 
By construction, for $x\in C$, each pair $(y_0,y_1)$ with $y_0 \in [x]_{F_0}$ and $y_1\in [b.x]_{F_1}$ satisfies $\rho^\bullet(y_0) < \rho^\bullet(y_1)$.
Hence $\rho^\bullet(y) < \rho^\bullet(\phi(y))$ for all $y$, so the mass transport principle applied to the indicator function of the graph of $\phi$ shows that $\mu (X)=\mu (\dom(\phi ))<\mu (\im(\phi ))=\mu (X)$, a contradiction. 
\end{example}

%% file: nontreeability.tex
\section{Applications to nontreeability}\label{section:non-treeability}

The contrapositive of \cref{nowhere_smooth_intersection} gives a criterion for nontreeability in the mcp setting, as illustrated by the following example.

\begin{example}
Take an mcp action of $\F_2 \times \Z$ whose orbit equivalence relation $E$ is nonamenable and for which the subrelation generated by $1\times \Z$ is nowhere smooth.
Then \cref{nowhere_smooth_intersection} implies that $E$ is nontreeable.
Indeed, denoting the free generators of $\F_2$ by $a$ and $b$, the subrelations generated by the subgroups $\gen{a} \times \Z$ and $\gen{b} \times \Z$ are each amenable and their intersection is nowhere smooth, but they together generate all of $E$. 
\end{example}

We generalize this example in \cref{normal_amenable=>anti-treeable} and give two nontreeability criteria for locally compact second countable groups in \cref{Hjorth_nontreeable,cor:noncomp_amen_norm_nontreeable}, the first of which leads to an alternative proof (and refinement) of a theorem of Hjorth \cite{Hjorth:non-treeability}.

\subsection{Class-bijective and relatively measure preserving extensions} 

\begin{defn}
Let $E$ and $F$ be mcp countable Borel equivalence relations on standard probability spaces $(X, \mu)$ and $(Y,\nu)$ respectively.
A pmp Borel homomorphism $\pi : (X,\mu ) \to (Y,\nu )$ from $E$ to $F$ is called \dfn{class-bijective} if for each $ x \in X$ the restriction $\pi \rest{[x]_E}$, of $\pi$ to $[x]_E$, is a bijection between $[x]_E$ and $[\pi(x)]_F$.
In this case we call $E$ a \dfn{class-bijective extension} of $F$ via $\pi$.
Given such a class-bijective extension, we have a natural embedding $[[F]]\hookrightarrow [[E]]$, $\psi\mapsto \tilde{\psi}$, satisfying $\pi \circ \tilde{\psi} = \psi \circ \pi$, and defined by taking $\tilde{\psi}(x)$ to be the unique $z\in [x]_E$ with $\pi (z)=\psi (\pi (x))$, for $\psi \in [[F]]$ and $x\in \dom (\tilde{\psi})\defeq \pi^{-1}(\dom (\psi ))$.
Such a class-bijective extension is called \dfn{relatively measure preserving} if the conditional expectation $\Phi_\pi : L^{\infty}(X ,\mu) \to L^{\infty}(Y,\nu)$ with respect to $\pi$ is $[[F]]$-equivariant.
\end{defn}

\begin{prop}
Let $E$ and $F$ be mcp countable Borel equivalence relations on standard probability spaces $(X, \mu)$ and $(Y,\nu)$ respectively. 
Suppose that $E$ is a class-bijective extension of $F$ via the pmp Borel homomorphism $\pi : (X,\mu )\to (Y,\nu )$.
Then the extension is relatively measure preserving if and only if 
the Radon--Nikodym cocycle of $E$ with respect to $\mu$ factors through $(\pi \times \pi)\rest{E}$ to the Radon--Nikodym cocycle of $F$ with respect to $\nu$.
\end{prop}

\begin{proof}
Let $\mu = \int _Y \mu _y \, d\nu$ be the disintegration of $\mu$ over $\nu$ via $\pi$.
The extension being relatively measure preserving is equivalent to the statement that $\tilde{\psi} _*\mu _{\psi ^{-1}(y)} = \mu _{y}$ for a.e.\ $y\in \mathrm{im}(\psi)$, for each $\psi \in [[F]]$.  
The statement about Radon--Nikodym cocycles is equivalent to the statement that $\frac{d\psi _*\nu}{d\nu}\circ \pi =\frac{d\tilde{\psi}_*\mu}{d\mu}$ for each $\psi \in [[F]]$. 

For each $\psi \in [[F]]$ and $f\in L^{\infty}(X,\mu )$ with support contained in $\mathrm{im}(\tilde{\psi})=\pi ^{-1}(\mathrm{im}(\psi ))$ we have

\begin{align}
\nonumber \int _{\mathrm{im}(\psi )} \int _{\pi ^{-1}(y)} f(x)\, d\tilde{\psi}_*\mu_{\psi ^{-1}(y)} (x)\, d\nu (y) 
&= 
\int _{\mathrm{im}(\psi )} \int _{\pi ^{-1}(y)} f(\tilde{\psi} (x)) \frac{d\psi ^{-1}_*\nu}{d\nu }(\pi (x))\, d\mu _y(x) \, d\nu (y)  
\\
\label{eqn:rel_mp} &= \int _{\mathrm{im}(\tilde{\psi})} f(x) \frac{d\psi _*\nu}{d\nu}(\pi (x))^{-1}\frac{d\tilde{\psi}_*\mu}{d\mu}(x)\, d\mu (x) .
\end{align}
If the extension is relatively measure preserving then the left hand side of \labelcref{eqn:rel_mp} is equal to $\int f \, d\mu$ for every such $f$, hence $\frac{d\psi _*\nu}{d\nu}\circ \pi = \frac{d\tilde{\psi}_*\mu}{d\mu}$ and the relationship between Radon--Nikodym cocycles is established.
Conversely, if the relationship between Radon--Nikodym cocycles is assumed, then the right hand side of \labelcref{eqn:rel_mp} is equal to $\int f \, d\mu$ for every such $f$, so uniqueness of disintegration implies $\tilde{\psi}_*\mu_{\psi ^{-1}(y)}=\mu _y$ for a.e.\ $y\in \mathrm{im}(\psi )$, and this holds for every $\psi \in [[F]]$, i.e., the extension is relatively measure preserving. 
\end{proof}

It is straightforward to see that if $E$ is a class-bijective extension of an amenable equivalence relation $F$, then $E$ is amenable as well.
Under the assumption that the extension is relatively measure preserving, the converse holds. 

\begin{lemma}\label{relatively-mp_ext}
Let $E$ and $F$ be mcp countable Borel equivalence relations on standard probability spaces $(X, \mu)$ and $(Y,\nu)$ respectively, and suppose that $E$ is a relatively measure preserving class-bijective extension of $F$ via $\pi$.
Then:
\begin{enumerate}[(a)]
\item \label{relatively-mp_ext:amenability} If $F$ is $\nu$-nowhere amenable, then $E$ is $\mu$-nowhere amenable.

\item \label{relatively-mp_ext:smoothness} If $F$ is $\nu$-nowhere smooth, then $E$ is $\mu$-nowhere smooth.
\end{enumerate}
\end{lemma}
\begin{proof}
For \labelcref{relatively-mp_ext:amenability}, towards the contrapositive, suppose that $A \subseteq X$ is a non-null $E$-invariant set on which $E$ is amenable, as witnessed by the unital, positive, $[[E \rest{A}]]$-equivariant linear map $P_A : L^{\infty}(E \rest{A}) \to L^{\infty}(A)$.
Because $A$ is $E$-invariant, the $[[F]]$-equivariance of the conditional expectation $\Phi_\pi : L^{\infty}(X ,\mu) \to L^{\infty}(Y,\nu)$ implies that $\Phi_\pi (1_A)$ is $F$-invariant.
Furthermore, after shrinking $A$ if necessary, we may assume that $\Phi _\pi (1_A)$ is bounded away from $0$ on its support, which we denote by $B$.
Let $h \mapsto \tilde{h}\defeq h \circ (\pi \times \pi)$ denote the inclusion of $L^{\infty}(F \rest{B})$ into $L^{\infty}(E\rest{\pi ^{-1}(B)})$.
The map $P_B: L^{\infty}(F \rest{B}) \to L^\infty(B)$ defined by $P_B(h) \defeq \Phi_\pi (P_A (\tilde{h}1_{E\rest{A}}))/\Phi_\pi (1_A)$ then witnesses the amenability of $F \rest{B}$. 

For \labelcref{relatively-mp_ext:smoothness}, let $\rho _E$ and $\rho _F$ denote the Radon--Nikodym cocycles of $E$ and $F$ respectively, so that $\rho _E=\rho _F\circ (\pi\times \pi)$. 
By \cref{smooth=rho-finite}, if $F$ is $\nu$-nowhere smooth then $\rho _F^y([y]_F)=\infty$ for $\nu$-a.e.\ $y\in Y$, hence $\rho _E^x([x]_E)=\rho _F^{\pi(x)}([\pi (x)]_F)=\infty $ for $\mu$-a.e.\ $x\in X$, so $E$ is $\mu$-nowhere smooth.
\end{proof}

\subsection{Treeability in the locally compact setting}
Let $G$ be a locally compact second countable group acting in a Borel way on a standard Borel space $X$. 
We call a Borel $E^X_G$-complete section $Y \subseteq X$ \dfn{locally countable} if its intersection with each $E^X_G$-class is countable.
In this context, locally countable Borel complete sections always exist by \cite{kechris:countable-sections}; earlier it was shown in \cite[Theorem 2.8]{Feldman-Hahn-Moore:orbit_structure} that such sections exist for mcp actions (where we are allowed to discard a $G$-invariant null set).

\begin{prop}\label{Borel_1_amenable_section}
Let $G$ be a locally compact second countable group, let $G\curvearrowright X$ be a Borel action of $G$ on a standard Borel space $X$, and let $Y\subseteq X$ be a locally countable Borel $E_G^X$-complete section.
If $G$ is amenable then $E_G^X\rest{Y}$ is Borel $1$-amenable. 
\end{prop}

\begin{proof}
Fix a left Haar measure $\mu _G$ on $G$.
Assuming $G$ is amenable, by \cite[Theorem G.3.1]{Bekka-delaHarpe-Valette} we may find a sequence $(f_k)_{k\in \N}$ of non-negative unit vectors in $L^1(G,\mu _G)$ with $\lim _{k\to\infty} \| g.f_k - f_k \| _1 = 0$ for all $g\in G$. 
By Luzin--Novikov uniformization we may find a Borel retraction $\varphi : X\to Y$ with $\varphi (x)\in G.x$ for all $x\in X$.
For each $k\in \N$ and $z\in Y$ define the probability vector $m^k_z\in \ell ^1([z]_{E^X_G\rest{Y}})$ by  
\[
m_z^k(y)\defeq \int _{\{ g\in G: \varphi (g^{-1}.z)=y \}} f_k (g)\, d\mu_G(g) .
\]
A standard computation verifies that the sequence $(m^k)_{k\in \N}$ provides a witness to Borel $1$-amenability of $E_G^X\rest{Y}$.
\end{proof}

\begin{lemma}\label{lcsc_treeable_permanence}
Let $G$ be a locally compact second countable group and let $G \actson X$ be a Borel action of $G$ on a standard Borel space $X$.

\begin{enumerate}[(a)]
\item\label{complete_sec_treeable} Let $Y \subseteq X$ be any Borel $E_G^X$-complete section.
Then $E^X_G$ is treeable if and only if $E^X_G \rest{Y}$ is treeable.

\item\label{subgroup_treeable} Let $H$ be a closed subgroup of $G$. 
If the action of $G$ on $X$ is free and if $E^X_G$ is treeable then $E^X_H$ is also treeable.
\end{enumerate}
\end{lemma}

\begin{proof}
For \labelcref{complete_sec_treeable}, by \cite{kechris:countable-sections}, we may find a locally countable Borel complete section $Y_0$ of $E^X_G$ that is contained in $Y$.
Thus, it is enough to prove that $E^X_G \rest{Y}$ is treeable if and only if $E^X_G \rest{Y_0}$ is treeable.
The forward direction of this is proved by an edge-sliding argument in \cite[Lemma 1.1]{Hjorth:non-treeability}, and the backward direction follows from Luzin--Novikov uniformization, as in \cite[(iv) of Proposition 3.3]{JKL}.

For \labelcref{subgroup_treeable}, let $Y \subseteq X$ be a locally countable Borel complete section for $E^X_H$.
By part \labelcref{complete_sec_treeable}, $E \defeq E^X_G \rest{Y}$ is treeable and we will be done if we show that $F \defeq E^X_H \rest{Y}$ is treeable.
Fix a Borel treeing $T$ of $E$ and similarly to \cite[Lemma 1.1]{Hjorth:non-treeability}, we consider the set
\[
C \defeq \set{(x,y) \in E : x \text{ is in the $T$-convex hull of } [y]_F}.
\]
This set is Borel by Luzin--Novikov uniformization, being the projection of the set of all triples $(x,y,z)$ of $E$-related points such that $(y,z) \in F$ and $x$ lies on the geodesic $[y,z]_T$.
Let $\tilde{F}$ be the countable Borel equivalence relation on $C$ defined by $[(x,y)]_{\tilde{F}}\defeq \set{(x, z) \in C : (y,z) \in F}$.
This equivalence relation is smooth as witnessed by the map $C \to Y \times (H\backslash G )$, $(x, g x) \mapsto (x,Hg)$, which is well-defined by the freeness of the action of $G$ on $X$.
From this and Luzin--Novikov uniformization we obtain a Borel function $f : C \to Y$ such that $f(x,y)$ is a point in $[y]_F$ of minimum $T$-distance to $x$ among the points in $[y]_F$ and $f(x,y) = f(x,z)$ whenever $(y,z) \in F$.
Now one proceeds as in the edge-sliding argument in the proof of \cite[(iii) Proposition 3.3]{JKL}.
\end{proof}

\begin{remark}\label{counterexample:nonfree}
The assumption of freeness of the action of $G$ is necessary in \cref{lcsc_treeable_permanence}\labelcref{subgroup_treeable}.
For example, for $i \in \set{ 0,1 }$ let $G_i$ be a copy of $\mathrm{PSL}_2(\R)$, let $\Gamma _i\leq G_i$ be a lattice, let $\Gamma$ be an irreducible lattice in the product $G=G_0\times G_1$, and let $\mu$ be the $G$-invariant Borel probability measure on $G/\Gamma$.
Then the orbit equivalence relation of $G$ acting on $G/\Gamma$ is (trivially) treeable since the action is transitive, but we claim that the orbit equivalence relation of the action $\Gamma _0\times \Gamma _1 \curvearrowright G/\Gamma$ is not treeable (even if we are allowed to discard a $\mu$-null set).
This follows from \cref{Hjorth_nontreeable} below, since a simple argument, using that each $G_i$ has the Howe--Moore property and acts ergodically on $(G/\Gamma ,\mu )$ and each $\Gamma _i$ has trivial center, shows that the action of $\Gamma _0\times \Gamma _1$ on $(G/\Gamma, \mu)$ is $\mu$-essentially free.


\end{remark}

\subsection{Anti-treeable configurations and nontreeability}

Let $G\curvearrowright (X,\mu )$ be an mcp action of a locally compact second countable group $G$ on a standard probability space $(X,\mu )$, and let $Y\subseteq X$ be a locally countable Borel $E_G^X$-complete section. 
By \cite[Theorem 6.17]{Ramsay:virtual_groups_and_actions}, $Y$ comes equipped with an $E_G^X\rest{Y}$-invariant measure class (of $\sigma$-finite Borel measures) having the same $E_G^X$-invariant null sets as $\mu$. 
This measure class is easily seen to be unique, and we call it \dfn{the measure class on $Y$ associated to $(E^X_G, \mu )$}.
Note that if the action of $G$ is pmp and $G$ is nonunimodular, then the measure class on $Y$ associated to $(E^X_G, \mu )$ may not contain an $E^X_G \rest{Y}$-invariant $\sigma$-finite Borel measure.

\begin{defn}
Let $E$ be an mcp countable Borel equivalence relation on a standard probability space $(Y,\nu )$. 
We call a pair $(Q,R)$ of Borel subrelations of $E$ an \dfn{anti-treeable configuration} for $E$ if the following conditions hold:
\begin{enumerate}[(i)]
\item $Q$ and $R$ are both $\nu$-amenable,
\item $Q\cap R$ is $\nu$-nowhere smooth,
\item the subrelation $Q \vee R$ generated by $Q$ and $R$ is $\nu$-nowhere amenable.
\end{enumerate}
\end{defn}

By \cref{nowhere_smooth_intersection}, if $E$ has an anti-treeable configuration then $E$ is $\nu$-nowhere treeable.

\begin{prop}\label{normal_amenable=>anti-treeable}
Let $\Gamma \actson (X,\mu)$ be an mcp action of a countable group $\Gamma$.
Assume the orbit equivalence relation $E_\Gamma$ of this action is $\mu$-nowhere amenable and that there is a normal subgroup $N$ of $\Gamma$ whose orbit equivalence relation $E_N$ is amenable and $\mu$-nowhere smooth.
Then $E_\Gamma$ has an anti-treeable configuration, and hence is not treeable.
\end{prop}

\begin{proof}
Since $E_N$ is amenable it is contained in some maximal amenable subrelation $Q$ of $E_\Gamma$.
We will define by measure theoretic exhaustion another amenable subrelation $R \subseteq E_\Gamma$ containing $E_N$, such that $(Q,R)$ is an anti-treeable configuration for $E$.

We recursively define a decreasing sequence $(X_n)_{n\geq 0}$ of Borel sets and an increasing sequence $(R_n)_{n\geq 0}$ of amenable subrelations of $E_\Gamma$ such that $R_n \rest{X_n} = E_N \rest{X_n}$ and $X_n$ is the essentially unique maximal set on which $S_n\defeq Q \vee R_n$ is amenable; in particular the set $X_n$ is $S_n$-invariant and $S_n\rest{X_n}=Q\rest{X_n}$.

Let $R_0 \defeq E_N$, so that $X_0\defeq X$ satisfies the base of the induction.
Supposing that $X_n$ and $R_n$ are defined.
If $X_n$ is null then take $X_{n+1} \defeq X_n$ and $R_{n+1} \defeq R_n$.
Otherwise, for each $\gamma \in \Gamma$, let
\[
A_n^\gamma \defeq \set{x \in X_n \cap \gamma^{-1} . X_n : (\gamma.x,x ) \notin Q},
\]
and let $\gamma_n \in \Gamma$ be such that $A_n \defeq A_n^{\gamma_n}$ has measure at least $\frac{1}{2} \sup_{\gamma \in \Gamma} \mu (A_n^\gamma)$.
Since $E_\Gamma$ is $\mu$-nowhere amenable and $X_n$ has positive measure, $A_n$ has positive measure as well.
Let $R_{n+1} \defeq R_n \vee \gamma_n\rest{A_n}$. 
Then $R_{n+1}\rest{X_n} = E_N\rest{X_n} \vee \gamma_n\rest{A_n}$, and $R_{n+1} \rest{X \setminus X_n} = R_n \rest{X\setminus X_n}$, so $R_{n+1}$ is $\mu$-amenable by \cref{amenable-coamenable} and the induction hypothesis.
Let $X_{n+1} \subseteq X$ be the essentially unique maximal set on which $S_{n+1} \coloneqq Q\vee R_{n+1}$ is amenable, so $X_{n+1}$ is $S_{n+1}$-invariant and disjoint from $A_n$. 
This and maximality of $Q$ imply $S_{n+1} \rest{X_{n+1}} = Q\rest{X_{n+1}}$ and $R_{n+1} \rest{X_{n+1}} = E_N\rest{X_{n+1}}$.

Let $R \defeq \bigcup_{n \in \N} R_n$, so $R$ is amenable.
To see that $(Q,R)$ is an anti-treeable configuration for $E$, it suffices to show that the set $X_\infty\defeq \bigcap_{n\geq 0}X_n$ is null. 
To this end, if $X_\infty$ has positive measure, then by the $\mu$-nowhere amenability of $E_\Gamma$, there is $\gamma \in \Gamma$ such that $A\defeq \{x \in X_\infty \cap \gamma^{-1} . X_\infty : (\gamma.x,x ) \notin Q\}$ has positive measure, contradicting the choice of $\gamma_n$ for $n$ large enough such that $\mu(A_n) < \frac{1}{2} \mu(A)$.
\end{proof}

We say that a locally compact second countable group $G$ \dfn{admits anti-treeable configurations} if there exists a free \emph{pmp} action of $G$ on a standard probability space $(Z,\eta )$, along with a locally countable Borel $E_G^Z$-complete section $Y$ such that $E_G^Z\rest{Y}$ has some anti-treeable configuration $(Q,R)$ as above (where $\nu$ is any probability measure from the measure class on $Y$ associated to $(E_G^Z, \eta )$). 
If a witness $(Q,R)$ to $G$ admitting anti-treeable configurations exists with the additional property that both $Q$ and $R$ are ergodic, then we say that $G$ \dfn{admits ergodic anti-treeable configurations}.

We recall that a locally compact second countable group $G$ is called \textbf{treeable} if $G$ admits some free pmp action whose orbit equivalence relation is treeable.

\begin{prop}\label{configuration_implies_nontreeability}
Let $G$ be a locally compact second countable group that contains a closed subgroup $H$ that admits anti-treeable configurations.
Then $G$ is not treeable.
\end{prop}

\begin{proof}
By the contrapositive of \cref{lcsc_treeable_permanence}\labelcref{subgroup_treeable}, it is enough handle the case when $G=H$, i.e., when $G$ itself admits anti-treeable configurations. 

Thus, there exists a free pmp action $G\curvearrowright (Z,\eta )$ of $G$, along with a locally countable Borel complete section $Y\subseteq Z$ for $E\defeq E_G^Z$, and an anti-treeable configuration $(Q,R)$ for $F\defeq E\rest{Y}$; here $Y$ is equipped with a Borel probability measure $\nu$ from the measure class on $Y$ associated to $(E, \eta )$.  
Given a free pmp action $G\curvearrowright (X,\mu )$, consider the product action of $G$ on $(\tilde{Z},\tilde{\eta})\defeq (X\times Z, \mu\times\eta )$ given by $g.(x,z)\defeq (g.x,g.z)$, and let $\tilde{E}$ denote the orbit equivalence relation of this action. 
The projection map $\tilde{Z}\to X$, $(x,z)\mapsto x$, is a class-bijective homomorphism from $\tilde{E}$ to $E^X_G$, so if $E^X_G$ were $\mu$-somewhere treeable, then $\tilde{E}$ would be $\tilde{\eta}$-somewhere treeable.
It therefore suffices to show that $\tilde{E}$ is $\tilde{\eta}$-nowhere treeable.

The set $\tilde{Y}\defeq X\times Y$ is a locally countable Borel complete section for $\tilde{E}$ and,
letting $\tilde{F}\defeq \tilde{E}\rest{\tilde{Y}}$, we see that the measure $\tilde{\nu}\defeq\mu\times\nu$ belongs to the measure class on $\tilde{Y}$ associated to $(\tilde{E}, \tilde{\eta})$. 
Moreover, $\tilde{F}$ is a relatively measure preserving class-bijective extension of $F$ via the projection map $\pi :\tilde{Y}\to Y$ to $Y$. 

Thus, the $\nu$-amenable subrelations $Q$ and $R$ of $F$ pull back under $(\pi\times \pi)\rest{\tilde{F}}$ to $\tilde{\nu}$-amenable subrelations $\tilde{Q}$ and $\tilde{R}$ of $\tilde{F}$, whose intersection is $\tilde{\nu}$-nowhere smooth, and which together generate a $\tilde{\nu}$-nowhere amenable subrelation by \cref{relatively-mp_ext}.
Hence $\tilde{F}$ is $\tilde{\nu}$-nowhere treeable by \cref{nowhere_smooth_intersection}.
It now follows that $\tilde{E}$ is $\tilde{\eta}$-nowhere treeable by \cref{lcsc_treeable_permanence}\labelcref{complete_sec_treeable}.
\end{proof}

\begin{theorem}\label{product_configuration}
Let $G_0$ and $G_1$ be locally compact second countable groups with $G_0$ nonamenable and $G_1$ noncompact.
Then $G_0\times G_1$ admits ergodic anti-treeable configurations.
\end{theorem}

\begin{proof}
For each $i\in \{ 0,1\}$ fix a free ergodic pmp action of $G_i$ on a standard probability space $(Z_i,\eta _i )$, let $E_i$ denote the orbit equivalence relation of this action, and
let $Y_i \subseteq Z_i$ be a locally countable Borel complete section for $E_i$. 
Let $F_i\defeq E_i\rest{Y_i}$, 
and let $\nu _i$ be a probability measure in the measure class on $Y_i$ associated to $(E_i,\eta _i)$.  

Since each $G_i$ is noncompact and the action on $(Z_i, \eta _i)$ is free, ergodic, and pmp, the relation $F_i$ is $\nu _i$-nowhere smooth and ergodic. 
Similarly, since $G_0$ is nonamenable, the relation $F_0$ is $\nu _0$-nowhere amenable \cite[Corollary 16]{Connes-Feldman-Weiss}.
We may therefore find two ergodic amenable subrelations $Q_0$ and $R_0$ of $F_0$ that together generate a $\nu _0$-nowhere amenable subrelation $Q_0\vee R_0$ of $F_0$. 
We may also find a $\nu _1$-nowhere smooth, $\nu _1$-amenable, ergodic subrelation $S_1$ of $F_1$. 

The product $Y\defeq Y_0\times Y_1$ is a locally countable Borel complete section for the equivalence relation $E\defeq E_0\times E_1$ generated by the product action of $G_0\times G_1$ on $(Z,\eta )\defeq (Z_0\times Z_1 , \eta _0\times\eta _1)$.
Then $Q\defeq Q_0\times S_1$ and $R\defeq R_0\times S_1$ are ergodic amenable subrelations of $F\defeq F_0\times F_1$ whose intersection is nowhere smooth for the product measure $\nu\defeq\nu_0\times\nu_1$, and $Q$ and $R$ together generate the $\nu$-nowhere amenable subrelation $(Q_0\vee R_0)\times S_1$ of $F$. 
Thus, the pair $(Q,R)$ witnesses that $G$ admits ergodic anti-treeable configurations.
\end{proof}

\begin{theorem}\label{normal_configuration}
Let $G$ be a nonamenable locally compact second countable group containing a closed amenable normal subgroup $H$ that is noncompact.
Then $G$ admits ergodic anti-treeable configurations.
\end{theorem}

\begin{proof}
Fix a free ergodic pmp action of $G/H$ on $(Z_0,\eta _0)$, which we also view as an action of $G$ (having kernel $H$), and fix also a free ergodic pmp action of $G$ on $(Z_1,\eta _1)$ for which the subgroup $H$ acts ergodically. 
Let $G$ act on $(Z,\eta )\defeq (Z_0\times Z_1,\eta _0\times \eta _1)$ diagonally.

Put $G_0\defeq G/H$ and $G_1\defeq H$.
For $i\in \{ 0, 1 \}$ let $E_i$ denote the orbit equivalence relation of the action of $G_i$ on $Z_i$, let $Y_i\subseteq Z_i$ be a locally countable Borel complete section for $E_i$, let $F_i\defeq E_i\rest{Y_i}$, and let $\nu _i$ be a probability measure in the measure class on $Y_i$ associated to $(E_i,\eta _i )$; it follows that $\nu _i$ is $F_i$-ergodic.
Then $Y\defeq Y_0\times Y_1$ is a locally countable Borel $E_G^Z$-complete section and the product measure $\nu\defeq\nu_0\times\nu_1$ is in the measure class on $Y$ associated to $(E_G^Z,\eta )$.

Since $G_1$ is amenable the relation $F_1$ is Borel $1$-amenable by \cref{Borel_1_amenable_section}, and since the action on $(Z_1, \eta _1)$ is free and pmp, the relation $F_1$ is $\nu _1$-nowhere smooth. 
Similarly, since $G_0$ is nonamenable, the relation $F_0$ is $\nu _0$-nowhere amenable \cite[Corollary 16]{Connes-Feldman-Weiss}.
We may therefore find two ergodic amenable subrelations $Q_0$ and $R_0$ of $F_0$ that together generate a $\nu _0$-nowhere amenable subrelation $Q_0\vee R_0$ of $F_0$. 

Let $F\defeq E_G^Z\rest{Y}$, let $p_0 : Y\to Y_0$ denote the projection map to $Y_0$, and let $\tilde{p}_0:F\to F_0$ denote the restriction of $p_0\times p_0$ to $F$.
Let $\Delta _{Y_0}$ denote the equality relation on $Y_0$, and observe that the relation $\tilde{p}_0^{-1}(\Delta _{Y_0})=\Delta _{Y_0}\times F_1$ is Borel 1-amenable and $\nu$-nowhere smooth.

\begin{claim}
The equivalence relations $Q\defeq \tilde{p}_0^{-1}(Q_0)$ and $R\defeq \tilde{p}_0^{-1}(R_0)$ on $(Y,\nu )$ are ergodic and $\nu$-amenable.
\end{claim}

\begin{pf}
Ergodicity of $Q$ and $R$ follows immediately from ergodicity of $F_1$ on $(Y_1,\nu _1)$ and ergodicity of $Q_0$ and $R_0$ on $(Y_0,\nu _0)$.
After discarding a $\nu_0$-null set from $Y_0$, the relation $Q_0$ is Borel 1-amenable, which implies that $\Delta _{Y_0} \times F_1$ is Borel 1-coamenable in $Q$.
Hence $Q$ is $\nu$-amenable by \cref{amenable-coamenable,2-amen_implies_mu-amen}.
Likewise, $R$ is $\nu$-amenable.
\end{pf}

\begin{claim}\label{claim:nowhere_amenable}
Let $S_0$ be a subrelation of $F_0$ and assume that the relation $S\defeq p_0^{-1}(S_0)$ is $\nu$-amenable. 
Then $S_0$ is $\nu _0$-amenable.
\end{claim}

\begin{pf}
Let $P:L^{\infty}(S)\to L^\infty (Y)$ be a unital, positive, $[[S]]$-equivariant linear map witnessing amenability of $S$.  
For each $h\in L^{\infty}(S_0)$ the function $h\circ\tilde{p}_0\in L^{\infty}(S)$ is $\Delta _{Y_0}\times F_1$-invariant, hence $P(h\circ \tilde{p}_0)$ is $\Delta _{Y_0}\times F_1$-invariant and, by ergodicity of $F_1$ on $(Y_1,\nu _1)$, it factors through $p_0$ to a function $P_0(h)\in L^{\infty}(Y_0)$.
Then the map $h\mapsto P_0(h)$ witnesses that $S_0$ is amenable.
\end{pf}

Since $Q\vee R = \tilde{p}_0^{-1}(Q_0\vee R_0)$, it follows that $Q\vee R$ is nowhere amenable: otherwise, since $Q\vee R$ is ergodic, $Q\vee R$ would be $\nu$-amenable, hence $Q_0\vee R_0$ would be $\nu _0$-amenable by \cref{claim:nowhere_amenable}, a contradiction.
The intersection of the relations $Q$ and $R$ is $\nu$-nowhere smooth since it contains $\Delta_{Y_0}\times F_1$, so the pair $(Q,R)$ is a witness to $G$ admitting ergodic anti-treeable configurations.
\end{proof}

\begin{cor}[{\cite[Corollary 0.6]{Hjorth:non-treeability}}]\label{Hjorth_nontreeable}
Let $G$ be a locally compact second countable group containing a closed subgroup isomorphic (as a topological group) to a product $G_0\times G_1$ where $G_0$ is nonamenable and $G_1$ is noncompact.
Then $G$ is not treeable.
\end{cor}

\begin{proof}
This is immediate from \cref{configuration_implies_nontreeability} and \cref{product_configuration}.
\end{proof}

\begin{remark}\label{rem:Hjorth_correction}
In the article \cite{Hjorth:non-treeability}, Hjorth deduces \cite[Corollary 0.6]{Hjorth:non-treeability} from \cite[Theorem 0.5]{Hjorth:non-treeability}, which states that if $G_0 \times G_1 \actson (X,\mu)$ is a pmp action of a product of locally compact groups such that $E^X_{G_0}$ is not $\mu$-amenable and the action of $G_1$ on almost all of its ergodic components is not essentially transitive, then $E^X_{G_0 \times G_1}$ is not treeable.
Unfortunately, as pointed out to us by Omer Tamuz, there is a counterexample to this statement, namely, the action of $G = G_0 \times G_1$ on $G / \Gamma$, where $G_0$, $G_1$, and $\Gamma$ are as in \cref{counterexample:nonfree}.

Examining the proof of \cite[Theorem 0.5]{Hjorth:non-treeability}, it appears that a mistake is made in the statement on page 401 that the set $A_1$ provides a Borel selector for a certain equivalence relation.
Nevertheless, we believe that the overall strategy of Hjorth's proof (which makes use of amenable boundary actions) can be adapted to provide a proof of \cite[Corollary 0.6]{Hjorth:non-treeability} and therefore of \cref{Hjorth_nontreeable}. 
This approach to proving \cref{Hjorth_nontreeable} is rather distinct from ours, although both approaches fundamentally rely on similar amenability and end selection arguments.
\end{remark}

\begin{cor}\label{cor:noncomp_amen_norm_nontreeable}
Let $G$ be a locally compact second countable group containing a closed noncompact amenable subgroup $H$ whose normalizer $N_G(H)$ is nonamenable.
Then $G$ is not treeable.
\end{cor}

\begin{proof}
This is immediate from \cref{configuration_implies_nontreeability} and \cref{normal_configuration}.
\end{proof}

In the upcoming article \cite{BSME} it is shown that the collection of all nonamenable, nonunimodular, treeable locally compact second countable groups forms a single measure equivalence class.
Our results in this section, combined with the structure theory of locally compact groups, now imply the following about groups in this class.

\begin{cor}\label{cor:compact_open}
Every locally compact second countable group that is nonamenable, nonunimodular, and treeable contains a compact open subgroup.
Thus, for every such group, its identity component is compact, and the quotient by this compact normal subgroup is a totally disconnected locally compact group.
\end{cor}

\begin{proof}
By \cite[Theorem~5.1]{BFS20} (see also \cite[Theorem~3.3.3]{BM02}), given any locally compact group $G$, there is an open finite index normal subgroup $H$ of $G$ containing the amenable radical $A$ of $G$ (i.e., $A$ is the largest closed normal amenable subgroup of $G$), such that $H/A$ is topologically isomorphic to a product $M \times L$, where $M$ is a totally disconnected locally compact group with trivial amenable radical, and $L$ is a connected, center-free, semisimple real Lie group without compact factors.

Assuming now that $G$ is second countable, nonamenable, nonunimodular, and treeable, by \cref{cor:noncomp_amen_norm_nontreeable} the amenable normal subgroup $A$ must be compact.
Since $H$ is open and of finite index in $G$, it follows that $H/A \cong M\times L$ is also nonamenable, nonunimodular, and treeable.
If $M$ and $L$ were both nontrivial, then $M$ would be nonamenable, since its amenable radical is trivial, and $L$ would be noncompact, since it has no compact factors; but this contradicts \cref{Hjorth_nontreeable}.
Thus, either $M$ or $L$ is trivial.
The group $L$ is unimodular by \cite[Corollary~8.31]{Knapp02}, whereas $H/A$ is not, so $M$ cannot be trivial.
It follows that $L$ is trivial and $H/A$ is totally disconnected. 
Since $H/A$ is open in $G/A$, the group $G/A$ is also totally disconnected.
By van Dantzig's theorem \cite{vD36}, $G/A$ contains a compact open subgroup, and since $A$ is compact, the preimage of such a subgroup under the quotient map $G\to G/A$ is a compact open subgroup of $G$.
\end{proof}

%% file: main.bbl
\begin{thebibliography}{GPTTDW26}

\bibitem[Ada88]{Adams:indec_treed_eq_rel_end-selection}
Scot Adams.
\newblock Indecomposability of treed equivalence relations.
\newblock {\em Israel J. Math.}, 64(3):362--380, 1988.
\newblock \href {https://doi.org/10.1007/BF02882427} {\path{doi:10.1007/BF02882427}}.

\bibitem[Ada90]{Adams90}
Scot Adams.
\newblock Trees and amenable equivalence relations.
\newblock {\em Ergodic Theory Dynam. Systems}, 10(1):1--14, 1990.
\newblock \href {https://doi.org/10.1017/S0143385700005368} {\path{doi:10.1017/S0143385700005368}}.

\bibitem[BdlHV08]{Bekka-delaHarpe-Valette}
Bachir Bekka, Pierre de~la Harpe, and Alain Valette.
\newblock {\em Kazhdan's property ({T})}, volume~11 of {\em New Mathematical Monographs}.
\newblock Cambridge University Press, Cambridge, 2008.
\newblock \href {https://doi.org/10.1017/CBO9780511542749} {\path{doi:10.1017/CBO9780511542749}}.

\bibitem[BFS20]{BFS20}
Uri Bader, Alex Furman, and Roman Sauer.
\newblock Lattice envelopes.
\newblock {\em Duke Math. J.}, 169(2):213--278, 2020.
\newblock \href {https://doi.org/10.1215/00127094-2019-0042} {\path{doi:10.1215/00127094-2019-0042}}.

\bibitem[BM02]{BM02}
Marc Burger and Nicolas Monod.
\newblock Continuous bounded cohomology and applications to rigidity theory.
\newblock {\em Geom. Funct. Anal.}, 12(2):219--280, 2002.
\newblock \href {https://doi.org/10.1007/s00039-002-8245-9} {\path{doi:10.1007/s00039-002-8245-9}}.

\bibitem[Bow18]{Bowen:eq_rel_act_hyperbolic}
Lewis Bowen.
\newblock Equivalence relations that act on bundles of hyperbolic spaces.
\newblock {\em Ergodic Theory Dynam. Systems}, 38(7):2447--2492, 2018.
\newblock \href {https://doi.org/10.1017/etds.2016.145} {\path{doi:10.1017/etds.2016.145}}.

\bibitem[CFW81]{Connes-Feldman-Weiss}
A.~Connes, J.~Feldman, and B.~Weiss.
\newblock An amenable equivalence relation is generated by a single transformation.
\newblock {\em Ergodic Theory Dynam. Systems}, 1(4):431--450, 1981.
\newblock \href {https://doi.org/10.1017/s014338570000136x} {\path{doi:10.1017/s014338570000136x}}.

\bibitem[CK18]{Chen-Kechris}
Ruiyuan Chen and Alexander~S. Kechris.
\newblock Structurable equivalence relations.
\newblock {\em Fund. Math.}, 242(2):109--185, 2018.
\newblock \href {https://doi.org/10.4064/fm428-7-2017} {\path{doi:10.4064/fm428-7-2017}}.

\bibitem[FHM78]{Feldman-Hahn-Moore:orbit_structure}
Jacob Feldman, Peter Hahn, and Calvin~C. Moore.
\newblock Orbit structure and countable sections for actions of continuous groups.
\newblock {\em Adv. in Math.}, 28(3):186--230, 1978.
\newblock \href {https://doi.org/10.1016/0001-8708(78)90114-7} {\path{doi:10.1016/0001-8708(78)90114-7}}.

\bibitem[FM77]{Feldman-Moore}
Jacob Feldman and Calvin~C. Moore.
\newblock Ergodic equivalence relations, cohomology, and von {N}eumann algebras. {I}.
\newblock {\em Trans. Amer. Math. Soc.}, 234(2):289--324, 1977.
\newblock \href {https://doi.org/10.2307/1997924} {\path{doi:10.2307/1997924}}.

\bibitem[GPTTDW26]{BSME}
Damien Gaboriau, Antoine Poulin, Anush Tserunyan, Robin Tucker-Drob, and Konrad Wróbel.
\newblock The measure equivalence classification of {B}aumslag--{S}olitar groups, 2026.
\newblock \href {https://arxiv.org/abs/2610.34323} {\path{arXiv:2610.34323}}.

\bibitem[Hjo08]{Hjorth:non-treeability}
Greg Hjorth.
\newblock Non-treeability for product group actions.
\newblock {\em Israel J. Math.}, 163:383--409, 2008.
\newblock \href {https://doi.org/10.1007/s11856-008-0017-2} {\path{doi:10.1007/s11856-008-0017-2}}.

\bibitem[JKL02]{JKL}
S.~Jackson, A.~S. Kechris, and A.~Louveau.
\newblock Countable {B}orel equivalence relations.
\newblock {\em J. Math. Log.}, 2(1):1--80, 2002.
\newblock \href {https://doi.org/10.1142/S0219061302000138} {\path{doi:10.1142/S0219061302000138}}.

\bibitem[Kai10]{Kaimanovich:Hopf_decomp}
Vadim~A. Kaimanovich.
\newblock Hopf decomposition and horospheric limit sets.
\newblock {\em Ann. Acad. Sci. Fenn. Math.}, 35(2):335--350, 2010.
\newblock \href {https://doi.org/10.5186/aasfm.2010.3522} {\path{doi:10.5186/aasfm.2010.3522}}.

\bibitem[Kec92]{kechris:countable-sections}
Alexander~S. Kechris.
\newblock Countable sections for locally compact group actions.
\newblock {\em Ergodic Theory Dynam. Systems}, 12(2):283--295, 1992.
\newblock \href {https://doi.org/10.1017/S0143385700006751} {\path{doi:10.1017/S0143385700006751}}.

\bibitem[Kec95]{bible}
Alexander~S. Kechris.
\newblock {\em Classical descriptive set theory}, volume 156 of {\em Graduate Texts in Mathematics}.
\newblock Springer-Verlag, New York, 1995.
\newblock \href {https://doi.org/10.1007/978-1-4612-4190-4} {\path{doi:10.1007/978-1-4612-4190-4}}.

\bibitem[KEOSS24]{EOSS:hyperfiniteness}
Srivatsav Kunnawalkam~Elayavalli, Koichi Oyakawa, Forte Shinko, and Pieter Spaas.
\newblock Hyperfiniteness for group actions on trees.
\newblock {\em Proc. Amer. Math. Soc.}, 152(9):3657--3664, 2024.
\newblock \href {https://doi.org/10.1090/proc/16851} {\path{doi:10.1090/proc/16851}}.

\bibitem[KL16]{Kerr-Li:book}
David Kerr and Hanfeng Li.
\newblock {\em Ergodic theory}.
\newblock Springer Monographs in Mathematics. Springer, Cham, 2016.
\newblock Independence and dichotomies.
\newblock \href {https://doi.org/10.1007/978-3-319-49847-8} {\path{doi:10.1007/978-3-319-49847-8}}.

\bibitem[KM04]{Kechris-Miller}
Alexander~S. Kechris and Benjamin~D. Miller.
\newblock {\em Topics in orbit equivalence}, volume 1852 of {\em Lecture Notes in Mathematics}.
\newblock Springer-Verlag, Berlin, 2004.
\newblock \href {https://doi.org/10.1007/b99421} {\path{doi:10.1007/b99421}}.

\bibitem[Kna02]{Knapp02}
Anthony~W. Knapp.
\newblock {\em Lie groups beyond an introduction}, volume 140 of {\em Progress in Mathematics}.
\newblock Birkh\"{a}user Boston, Inc., Boston, MA, second edition, 2002.

\bibitem[Lev95]{Levitt}
Gilbert Levitt.
\newblock On the cost of generating an equivalence relation.
\newblock {\em Ergodic Theory Dynam. Systems}, 15(6):1173--1181, 1995.
\newblock \href {https://doi.org/10.1017/S0143385700009846} {\path{doi:10.1017/S0143385700009846}}.

\bibitem[Mil08]{Miller:meas_with_cocycle_II}
Benjamin Miller.
\newblock The existence of measures of a given cocycle. {II}. {P}robability measures.
\newblock {\em Ergodic Theory Dynam. Systems}, 28(5):1615--1633, 2008.
\newblock \href {https://doi.org/10.1017/S0143385707001125} {\path{doi:10.1017/S0143385707001125}}.

\bibitem[Ram08]{Ramsay:virtual_groups_and_actions}
Arlan Ramsay.
\newblock Virtual groups for group representations.
\newblock In {\em Group representations, ergodic theory, and mathematical physics: a tribute to {G}eorge {W}. {M}ackey}, volume 449 of {\em Contemp. Math.}, pages 379--398. Amer. Math. Soc., Providence, RI, 2008.
\newblock \href {https://doi.org/10.1090/conm/449/08720} {\path{doi:10.1090/conm/449/08720}}.

\bibitem[TTD]{TTD:paddleball}
Anush Tserunyan and Robin Tucker-Drob.
\newblock Paddleball with {C}arri{\`e}re and {G}hys.
\newblock Forthcoming.

\bibitem[TTD25]{T-TD:mcp_Adams}
Anush Tserunyan and Robin Tucker-Drob.
\newblock The {R}adon--{N}ikodym topography of acyclic measured graphs, 2025.
\newblock \href {https://arxiv.org/abs/2512.24741} {\path{arXiv:2512.24741}}.

\bibitem[vD36]{vD36}
D.~van Dantzig.
\newblock Zur topologischen {Algebra}. {III}. {Brouwersche} und {Cantorsche} {Gruppen}.
\newblock {\em Compositio Math.}, 3:408--426, 1936.
\newblock URL: \url{https://www.numdam.org/item/CM_1936__3__408_0/}.

\end{thebibliography}
